\documentclass[12 pt, reqno]{amsart}
\title{Positivity of the Cotangent Bundle of Complex Hyperbolic Manifolds with Cusps}
\author{Soheil Memariansorkhabi}
\address{Dept. of Mathematics, University of Toronto, Toronto, Canada.}
\email{soheil.memarian@mail.utoronto.ca}

\usepackage[utf8]{inputenc}

\usepackage{amsfonts}

\usepackage{mathrsfs}
\usepackage[utf8]{inputenc}
\usepackage{amsmath, amsfonts, amsthm}
\usepackage{amssymb}
\usepackage{fullpage}
\usepackage{mathtools}
\usepackage{enumitem}
\usepackage{amsmath}
\usepackage{hyperref}
\usepackage[dvipsnames]{xcolor}
\usepackage[all]{xy}
\usepackage{tikz-cd}

\hypersetup{
    colorlinks=true,
    citecolor=blue,
    linkcolor=red,
    filecolor=magenta,      
    urlcolor= cyan,
    }

\newcommand{\tildeb}[1]{\stackrel{\sim}{\smash{#1}\rule{0pt}{1.1ex}}}

\newtheorem{theorem}{Theorem}
\newtheorem*{theorem*}{Theorem}

\newtheorem{Lemma}[theorem]{Lemma}
\newtheorem{Proposition}[theorem]{Proposition}
\newtheorem{Corollary}[theorem]{Corollary}
\theoremstyle{definition}
\newtheorem{Remark}[theorem]{Remark}
\newtheorem{definition}[theorem]{Definition}

\newcommand{\R}{\mathbb{R}} 
\newcommand{\Z}{\mathbb{Z}} 
\newcommand{\Q}{\mathbb{Q}} 
\newcommand{\C}{\mathbb{C}} 
\newcommand{\Sn}{\mathbb{S}}
\newcommand{\He}{\mathcal{H}}
\newcommand{\Proj}{\mathbb{P}}
\newcommand{\Xb}{\overline{X}}
\newcommand{\dis}{\displaystyle}
\newcommand{\del}{\partial}
\newcommand{\delb}{\bar{\partial}}
\newcommand{\B}{\Big}
\newcommand{\bi}{\big}
\newcommand{\Cot}{\Omega^{1}_{\Xb}}
\newcommand{\lgcot}{\Omega^{1}_{\Xb}\big(\log(D)\big)}
\newcommand{\Yb}{\overline{Y}}
\newcommand{\la}{\langle}
\newcommand{\ra}{\rangle}

\begin{document}

\begin{abstract}
Let $\Xb$ be the toroidal compactification of a cusped complex hyperbolic manifold $X=\mathbb{B}^n/\Gamma$ with the boundary divisor $D=\overline{X}\setminus X$. The main goal of this paper is to find the positivity properties of $\Omega^{1}_{\Xb}$ and $\lgcot$ depending intrinsically on $X$. We prove that $\lgcot \la -r D \ra$ is ample for all sufficiently small rational numbers $r >0$, and $\lgcot$ is ample modulo $D.$ Further,  we conclude that if the cusps of $X$ have uniform depth greater than $4\pi$, then $\Omega^{1}_{\Xb}$ is semi-ample and is ample modulo $D$, all subvarieties of $X$ are of general type, and every smooth subvariety $V\subset \Xb$ intersecting $X$ has ample $K_{V}$. Finally, we show that 
the minimum volume of subvarieties of $\Xb$ intersecting both $X$ and $D$ 
tends to infinity in cofinal towers of normal covering of $X.$ 

\end{abstract}

\maketitle

\section{Introduction}

Let $\He$ be a bounded symmetric domain and $X:=\Gamma \backslash \He$ be a quotient of $\He$ by a torsion-free lattice $\Gamma\subset \operatorname{Aut}(\He)$.  Various hyperbolicity properties of a toroidal compactification $\Xb$ of $X$ and curves in $\Xb$ have been proved by studying the positivity of $\Q$-line bundles $K_{\Xb}(D)-r D$, where $D$ is the boundary divisor, $K_{\Xb}(D)$ is the log-canonical bundle and $r$ is a number varying on an interval, the size of which depends on $X$. For instance, see \cite{nadel1989nonexistence, shepherd2005perfect} for the case of moduli space of principally polarized abelian varieties $\mathcal{A}_{g},$ \cite{bakker2018geometric} for the case of Hilbert modular varieties 
and \cite{bakker2018kodaira,di2015effective} 
for the case of ball quotients.

\renewcommand*{\thetheorem}{\Alph{theorem}}

Motivated by these results for the twists of $K_{\Xb}(D)$ and their implications for the hyperbolicity properties, we study the positivity of the twists of $\lgcot$ and $\Cot$ in the form of the $\Q$-vector bundles $\lgcot \la - r D \ra$ and $\Cot \la - r D \ra$. As the vector bundles $\lgcot$ and $\Cot$ restrict well to subvarieties (of arbitrary dimensions), the positivity of these bundles reveals various hyperbolicity properties of subvarieties of $\Xb$.   

\renewcommand*{\thetheorem}{\Alph{theorem}}

In this paper, we focus on this problem for a complex ball quotient $X=\Gamma \backslash\mathbb{B}^n,$ where $\Gamma \subset \operatorname{Aut}(\mathbb{B}^n)$ is a torsion-free lattice. We refer to $X$ as a complex hyperbolic manifold with cusps, where cusps correspond to parabolic fixed points of $\Gamma$. For cusps of $X$ we used a uniform depth $d$ to measure the size of the largest embedded cusp neighborhoods which are disjoint from each other (see \ref{UniDep}).    
We know that the uniform depth of cusps is at least $2$ and it tends to infinity in the cofinal towers of normal covering (see Lemma \ref{tower}).

Our first main result is the following theorem:
\begin{theorem} \label{SlopLog}
Let $X$ be a complex hyperbolic manifold with cusps whose toroidal compactification $\Xb$ has no orbifold points, and $d$ be the uniform depth of cusps. Then, the $\Q$-vector bundle
$$\lgcot \la -r D \ra$$
is ample for all rational $r \in (0,d/4\pi)$.
\end{theorem}

It was known that for small enough $\epsilon>0,$ the $\Q$-vector bundle $\Omega^{1}_{\overline{X}}\big(\log(D)\big) \la - \epsilon D \ra$
is big. Combining this with the observation of Mumford \cite[section 4]{mumford1977hirzebruch}, Brunebarbe(\cite{brunebarbe2016strong}) and Cadorel(\cite{Cad}) proved that for a cover $X'\to X$ such that the ramification index of each cusp is at least $l$, the $\Q$-vector bundle $\Omega^{1}_{\overline{X'}}\big(\log(D')\big) \la - l\epsilon D' \ra$
is big, where $\overline{X'}$ is the toroidal compactification of $X'$, and $D'$ is the boundary divisor.  By assuming that the ramification index $l$ is sufficiently large, they conclude the bigness of the cotangent bundle $\Omega^1_{\overline{X'}}$ on the cover $X'$(see \cite[Theorem 4]{Cad} and \cite[sec 6.2]{brunebarbe2016strong}).      

 Theorem \ref{SlopLog}, however, not only has stronger positivity conclusion for the twisted bundle, but also it has an advantage that it implies the positivity result for $X$ itself rather than some cover of $X.$ Notably, in Theorem \ref{SlopLog}, the quantity $r$ can be chosen to be a rational number close to $1/2\pi$ for all $X,$ since $d$ is always greater than or equal to $2.$ In particular, Theorem \ref{SlopLog}  can be applied when $X$ is not a cover of other locally symmetric domains, i.e, in the case that $\Gamma$ is a maximal lattice in $\operatorname{Aut}(\mathbb{B}^n)$(for every even $n,$ there exits infinitely many such lattices, see for example \cite{emery2014covolumes}).


Assuming that the uniform depth of cusps is sufficiently large, we showed that $\Cot$ is semi-ample and is ample modulo $D:$  

\begin{theorem} \label{SemiCot} With $X$ and $\Xb$ as in Theorem \ref{SlopLog}, suppose that the uniform depth of cusps is greater than $4\pi$. Then, the cotangent bundle $\Omega^{1}_{\Xb}$ is semi-ample and is ample modulo $D$. 
\end{theorem}
It is remarkable that for some towers of coverings, the uniform depth of cusps can increase while the ramification index remains unchanged. For instance, the uniform depth of cusps of the classical modular curve $X_0(p)$ tends to infinity as $p$ approaches infinity; however, $X_0(p) \to X(1)$ has a cusp with a ramification index of $1$ for every $p$ (see Remark \ref{modularcurve} for details). Such towers play a crucial role in number theory, as seen in Wiles's proof of Fermat's Last Theorem (see \cite{wiles1995modular}) and Mazur's torsion theorem (see \cite{mazur1978rational}). Higher-dimensional analogues of these towers are natural objects to study for generalizing these theorems (see \cite{bakker2018geometric}). Proving the hyperbolicity properties of these towers, which have a cusp with ramification index $1$, is not within reach with the methods used in \cite{brunebarbe2016strong, brunebarbe2020increasing, Cad}.

Semi-ampleness of $\Cot$ in Theorem \ref{SemiCot} implies that symmetric differentials over $\Xb$ forms a finitely generated $\C$-algebra(see \cite[Example 2.1.29]{lazarsfeld1}):  
\begin{Corollary}\label{fg}
With the same assumptions as Theorem \ref{SemiCot}, the graded ring $$\bigoplus_{n\ge 0}H^{0}(\Xb,S^{n}\Omega^1_{\Xb})$$ is finitely generated $\C$-algebra.  
\end{Corollary}

In Theorem \ref{SlopLog}, $\lgcot$ can not be replaced by $\Cot$ without adding any extra assumptions about $X$ because there are examples due to Hirzebruch \cite{hirzebruch1984chern} of the toroidal compactification of a 2-dimensional ball quotient whose canonical bundle is not even nef. We showed that assuming the uniform depth is sufficiently large the replacement is possible as long as $r$ varies on a smaller interval. 

\begin{theorem}\label{CotPositivity} With $X$ and $\Xb$ as in Theorem \ref{SlopLog}, suppose that the uniform depth of cusps $d$ is greater than $4\pi$. Then, the $\Q$-vector bundle
$$\Omega^{1}_{\Xb}\la -r D \ra$$ 
is ample for all rational $r \in (0,-1+d/4\pi)$.
\end{theorem}

As the cotangent bundle is well-behaved under the restriction to subvarieties, by virtue of Theorem \ref{CotPositivity} we get hyperbolicity of subvarieties of an arbitrary dimension in the following sense:

\begin{Corollary} \label{HyperVarIn}
With the same notations as in Theorem \ref{SlopLog}, suppose $V$ is a smooth subvariety of $\Xb$ intersecting $X$ with dimension $m>0.$ Then, 
$$K_V-(r -1) D_{|V}$$
is ample for all rational $r \in (0,\frac{dm}{4\pi}).$ In particular, if $m> 6$ or $d>4\pi$, then $K_{V}$ is ample. 
\end{Corollary}

\begin{Corollary}\label{GenSubInt}
With the same $X$ as Theorem \ref{SlopLog}, all subvarieties of $X$ are of general type provided that the uniform depth of cusps is greater than $4 \pi.$ 
\end{Corollary}

As explained after Theorem \ref{SlopLog}, Brunebarbe and Cadorel proved the bigness of $\Omega^1_{\overline{X'}},$ with the assumption that $X'$ is a highly ramified cover of $X.$ Then, they applied the difficult theorem of Campana-P{\u{a}}un's \cite{campana2015orbifold} to conclude that all subvarieties of $X'$ are of general type. Beside the fact that Corollary \ref{GenSubInt} tells us strong hyperbolicity of $X$ itself, rather than cover of $X,$ our approach has another advantage that it bypass the theorem of Campana-P{\u{a}}un to prove all subvarities of $X$ are of general type because of nice restriction properties of ampleness.              

Bakker and Tsimerman \cite{bakker2018kodaira} proved that $K_{\Xb}(D)-r D$ is ample for $r \in (0,(n +1)/2)$, where $n= dim \ X$ and concluded that if $n \ge 6$, $K_{\Xb}$ is ample. Similarly, Cadorel \cite{cadorel2021subvarieties} showed that $V\subset \Xb$ with $dim \ V \ge 6$ is of general type if $V \not \subset D$.


In general, for a smooth variety with dimension greater than 1, the log-cotangent bundle is never ample as its restriction to the boundary is an extension of the trivial bundle.(\cite{Ascher}, \cite{Brotbek}). For a toroidal compactification of a locally symmetric domain, Brunebarbe \cite{brunebarbe2018symmetric} proved that $\lgcot$ is big. In the case of ball quotients,  
Theorem \ref{SlopLog} demonstrates that $\lgcot$ is a limit of ample $\Q$-vector bundles. Moreover, we can deduce from Theorem \ref{SlopLog} the following Corollary:
\begin{Corollary} \label{logPositivity}
With $X$ and $\Xb$ as in Theorem \ref{SlopLog},
$\lgcot$ is ample modulo $D$ and  nef.
\end{Corollary}

In particular, Corollary \ref{logPositivity} implies that
$\lgcot$ is big and nef, which Cadorel also recently proved (\cite{Cad}). He used the positivity of the Bergman metric on the open part together with a result of \cite{boucksom2002volume} to show that the volume of the line bundle $O(1)$ over $\Proj(\Cot)$ must be positive which implies that $O(1)$ is big. Similarly, he proved that $O(1)$ over $\Proj(\Cot)$ is nef by showing that the volume of the pull back of $O(1)$
over $\Proj(\Cot)$ to every subvariety of $\Xb$ not contained in $D$ is positive.    

The volume of a line bundle $L$ on a projective variety $V$ of dimension $m$ is defined to the non-negative real number 
$$\operatorname{vol}_{V}(L):=\limsup\limits_{k\rightarrow \infty}\dfrac{h^0(V, L^{\otimes k})}{k^m/m!}$$
which measures the asymptotic growth rate of the pluri-canonical linear series $|kL|.$

 Brunebarbe \cite{brunebarbe2020increasing} used the volume of the canonical bundle of a variety to measure its hyperbolicity. In the setting of Shimura varieties (actually in more general setting, see \cite[Theorem 1.1.]{brunebarbe2020increasing}), he proved that the minimum volume of subvarieties of $Y(p),$ $Y$ with full level structure, tends to infinity.  As a generalization of Brunebarbe's work \cite{brunebarbe2020increasing}, for cofinal towers of normal coverings of ball quotients we prove the following Corollary:

\begin{Corollary} \label{VolIn} Let $X$ be a complex hyperbolic manifold with cusps and $\{X_i\}_{i=1}^{\infty}$ be a cofinal tower of normal coverings of $X=X_1.$ 
 Then, given a positive number $v$, for all but finitely many $i,$ every subvariety $V$ of $X_i$ containing a cusp of $X_i$ have $vol(K_V)>v.$ 
\end{Corollary}

Indeed, we show that the volume of the canonical bundle (and the log-canonical bundle) of subvarieties of $X$ containing a cusp of $X,$ is controlled by the uniform depth of cusps (see Corollary \ref{Vol} and Corollary \ref{CanonicalVol}).

Using results of $L^2$-estimates for a K\"ahler manifold, Wong \cite{Wong} proved that there exist a constant $r (n)$ and $d(n)$ depending only on the dimension of $X$, such that if the injectivity radius is larger than $r(n)$ and the uniform depth of cusps is larger than $d(n)$, then $\Cot$ is ample modulo $D$. His method relies on the existence theorem of $L^2$-estimation (\cite{andreotti1965carleman,MR179443}).

Theorem \ref{SlopLog} and Theorem \ref{SemiCot} substantially improve the previous results of Cadorel and Wong for the positivity of $\lgcot$ and $\Cot$, and our proofs are independent from their proofs.

Our method to prove Theorem \ref{SlopLog} is to construct a singular hermitian metric on $L:=O(1)$ over $\Proj(\lgcot)$ with the Bergman metric $h$ induced on $L$ and a carefully chosen weight function. To obtain the positivity of the twisted bundle, there has to be a trade-off between the positivity of $h$ and the weight function. This trade-off comes from the inequality we proved between $c_1(L_{|\pi^{-1}(X)},h)$ and a positive $(1,1)$-form   
 (see proofs of Proposition \ref{Positve intersection} and Theorem \ref{PositveLog}). The strength of our method is that it gives a criteria for the positivity of $\Cot$ explicitly and intrinsically depending on $X.$ This criteria can be checked for a given $X.$

We prove the semi-ampleness of $\Cot$ in Theorem \ref{SemiCot} by a hybrid
technique using both the results coming from the differential geometry of the toroidal compactification\textemdash namely, the positivity of the twisted cotangent bundle\textemdash and the results coming from the algebraic geometry of the toroidal compactification.   
(see proof of Proposition \ref{NefBoundry} and Theorem \ref{CotPos} for more details).


\numberwithin{theorem}{section}

\section*{Outline of the paper} In \textsection \ref{back}, we describe the toroidal compactification of the ball quotients and necessary background on the Siegel domain. In \textsection \ref{Sec2}, we study the properties of the Bergman metric induced on $O_{\Proj(\Omega^1_{\Xb}(\log(D))}(1).$ In \textsection \ref{Sec4}, we prove some basic lemmas about $\Q$-vector bundles. In \textsection \ref{sec3}, we prove the results related to the positively of $\lgcot$ including Theorem \ref{SlopLog} and Corollary \ref{logPositivity}. In \textsection \ref{sec3}, we prove the results related to the positivity of $\Cot$, namely Theorem \ref{SemiCot} and Theorem \ref{CotPositivity}. 
Finally, in \textsection \ref{sec3}, we deduce the applications to hyperbolicity of subvarieties in Corollary \ref{HyperVarIn}, Corollary \ref{GenSubInt} and Corollary \ref{VolIn}.

\section*{Acknowledgements}
I would like to thank my advisor Jacob Tsimerman
for introducing me to this problem, and for his constant support, as well as many insightful discussions. Also, I would like to thank Steven Lu and Luca Di Cerbo for their comments on an earlier draft of this paper.

\section{Background} \label{back}
\subsection{Complex Unit Ball}
The complex unit ball 
$$\mathbb{B}^n= \{z\in \C^n\ |\ |z|^2<1\}$$
has holomorphic automorphism group $G:=\operatorname{PU}(n,1)$. Let $\Gamma \subset G$ be a cofinite-volume discrete group which does not have any elliptic element. Therefore, $X:= \mathbb{B}^n/ \Gamma$ is a complex manifold with cusps, where cusps correspond to conjugacy class of maximal parabolic subgroups of $\Gamma$.

The Baily-Borel compactification \cite{baily1966compactification, siu2016compactification} of $X$ is obtained by adding a finite set of points in corresponding to the equivalences
classes of the rational boundary components of $\mathbb{B}^n$ under the action of $\Gamma.$
This compactification is a normal projective variety and we will refer to the the points added to $X$ as cusps of $X$.

The Bergman metric on $\mathbb{B}^n$ is the $\operatorname{PU}(n,1)-$invariant hermitian metric with holomorphic sectional curvature of $-1$ given by

$$h=4.\frac{(1-|z|^2).
\sum_{i}dz_i \otimes d\bar{z_i}+\sum_{i}\bar{z_i}dz_i \otimes \sum_{i}z_i d\bar{z}_i}{(1-|z|^2)^2}.$$ 

\subsection{Siegel Domain}
To understand the behaviour of Bergman metric around cusps, we found it more convenient to pass to the Siegel model $\Sn$ which is diffeomorphic to $\mathbb{B}^n$ and described by the holomorphic coordinates $(\zeta, z) $ of $\C^{n-1} \times \C$ in the following way:

$$\Sn:=\{(\zeta,z)\in \C^{n-1} \times \C \ | \ \operatorname{Im}(z)>|\zeta|^2 \}.$$

Let $u:= Im(z)-|\zeta|^2$ be the height coordinate on $\Sn$. The Kähler form of the Bergman metric on $\Sn$ is given by 
\begin{align}
    w_{\Sn}:= -2i\del\delb\log(u)
\end{align}
(see, for example, \cite[Lemma 2.1]{bakker2018kodaira}), and more explicitly we can write $\omega_{\Sn}$ in terms of $(\zeta,z)$-coordinates

\begin{align}
  w_{\Sn}
=
&2iu^{-2} \cdot \sum_{i=1}^{n-1}\sum_{k=1}^{n-1}(u \delta_{ik}
+\zeta_i \bar{\zeta_k}) d\zeta_{k}  \wedge d\bar{\zeta_{i}}\\ \nonumber
&+
u^{-2}\cdot \sum_{j=1}^{n-1}(
\bar{\zeta_{j}}  d\zeta_j \wedge  d\bar{z}
-
\zeta_j  dz  \wedge d\bar{\zeta_{j}}
)\\ \nonumber
&-
\frac{i}{2}u^{-2} dz \wedge  d\bar{z},  
\end{align}
where $\delta_{ik}$ is the Kronecker delta function. Using holomorphic coordinates $\zeta$ and $z$ on $\Sn$, we can set a holomorphic coordinates $(\zeta, z, \xi, w)$ on $\Omega^1_{\Sn},$ where $\xi=(\frac{\del}{\del \zeta_1}, \frac{\del}{\del \zeta_2},..., \frac{\del}{\del \zeta_{n-1}})$ and $w=\frac{\del}{\del z}$. Using the frame $e_i=\xi_{i}$ for $1 \le i\le n-1$ and $e_{n}=w$, the hermitian matrix of the Bergman metric is  

$$[h(e_i, \bar{e_j})]=
2\begin{bmatrix}
u^{-1}+u^{-2}|\zeta_1|^2& u^{-2}\zeta_{1}\bar{\zeta}_{2}& ...&u^{-2}\zeta_{1}\bar{\zeta}_{n-1}&\frac{1}{2i}u^{-2}\zeta_1\\

u^{-2}\zeta_{2}\bar{\zeta_{1}}&u^{-1}+u^{-2}|\zeta_2|^2& ... & u^{-2}\zeta_{2}\bar{\zeta}_{n-1}&\frac{1}{2i}u^{-2}\zeta_2\\

...&...& ... & ... & ...\\

u^{-2}\zeta_{n-1}\bar{\zeta}_{1} & u^{-2}\zeta_{n-1}\bar{\zeta}_{2}&...&u^{-1}+u^{-2}|\zeta_{n-1}|^2&\frac{1}{2i}u^{-2}\zeta_{n-1}\\

-\frac{1}{2i}u^{-2}\bar{\zeta_1} & -\frac{1}{2i}u^{-2}\bar{\zeta_2} &...&-\frac{1}{2i}u^{-2}\bar{\zeta}_{n-1}&\frac{1}{4}u^{-2}
\end{bmatrix}.$$

Siegel model has a preferred cusp at infinity whose parabolic stabilizer group $G_{\infty}$ is generated by Heisenberg isometries $\operatorname{I}_{\infty}$ and a one-dimensional torus $T$. Heisenberg isometries consists of Heisenberg Rotations $\operatorname{U}(n-1)$ and Heisenberg translations $\mathfrak{N}$. The group of Heisenberg Rotations identifies with $\operatorname{U}(n-1)$ acting on $\zeta$-coordinates of $\Sn$ in the usual way and The group of Heisenberg translations $\mathfrak{N} \cong \C^{n-1} \times \mathbb{R}$ acts on $\Sn$ via

$$(\tau, t): (\zeta, z) \longrightarrow (\zeta', z'):=\bi(\zeta+\tau, z+t+i|\tau|^2 + 2i(\tau, \zeta)\bi). $$
In particular, the elements $(0,t)\in \mathfrak{N}$ will be referred to as a vertical translation by $t$, and the subgroup generated by $(0,t)$ in $G_\infty$ will be denoted by $\operatorname{V}_{\infty}$.
We use $(A,\tau,t)  \in \operatorname{U}(n-1) \ltimes \mathfrak{N}$ to denote the transformation acting by 
$$
(A,\tau,t)\cdot (\zeta, z) \longrightarrow \bi(A\zeta+\tau, z+t+i|\tau|^2+2i(A\zeta, \tau)\bi ), 
$$
where the standard positive definite hermitian form $(.,.)$ chosen on $\C^{n-1}$. Using the chain rule we can observe that the action of $(A,\tau,t)$ induced on the cotangent bundle $\Omega^{1}_{\Sn}$ is
\begin{align} \label{ActonCot}
(A,\tau,t).(\zeta,z,\xi, w)
=
\bi (A\zeta+\tau, z+t+i|\tau|^2+2i(A\zeta, \tau), A\xi+2i(A\xi, \tau), w \bi ).
\end{align}
The vertical translation $\operatorname{V}_{\infty}$ is the center of $G_{\infty}$ and the quotient $\operatorname{V}_{\infty}\backslash \operatorname{I}_{\infty}$ is isomorphic to the unitary transformation of 
$\C^{n-1}$. 
The remaining generators of $G_{\infty}$ is the one-dimensional torus $T$ which acts by scaling  $$(\zeta,z)\to (a\zeta, a^2 z).$$     

\subsection{Depth of cusps}
For $u_i>0,$ the horoball centered at the cusp at infinity with height $u_i$ is the open set
$$B(u_i):=\{(\zeta,z)\in \Sn \ | \ u>u_i\}.$$
Note that the height coordinate $u$ on $\Sn$ is invariant under the action of Heisenberg rotations $\operatorname{U}(n-1)$ and Heisenberg translations $\mathfrak{N}$ because $\operatorname{U}(n-1)$ fixes $|\zeta|$ and $z,$ and  $$\operatorname{Im}(z')-|\zeta'|^2= \operatorname{Im}(z)-|\zeta|^2.$$ It follows that the horoball is invariant with respect to the action of Heisenberg rotations and Heisenberg translations and the one-dimensional torus $T$ scales $B(u)$ to $B(au)$. 

\begin{definition}
For a cusp $c_i$, there exists a  $g\in G$ translating $c_i$ to $\infty$. Let $\Gamma'_{i}:=(g\cdot\Gamma\cdot g^{-1})_{\infty}$  be the parabolic stabilizer of $\infty.$ The smallest $u_i$ such that $\Gamma'_{i} \backslash B(u_i)$ injects into $X$ will be called the height of the $c_i$. We will refer to $H_i(u)=\Gamma'_{i} \backslash B(u)$ for $u>u_i$ as a horoball neighborhood of $c_i$.
\end{definition}

Thanks to the Shimuzu's lemma, $H_i(u_i)$ must injects into $X$ for a large enough $u_i$.

\begin{definition}\label{DepDef} Suppose $u_i$ is the height of cups $c_i,$ and $t_i$ is the shortest vertical translation in the parabolic stabilizer of $c_i$. The number $d_i=t_i/u_i$ is called the depth of $c_i$ which is invariant under conjugating the lattice $\Gamma.$
\end{definition}

\begin{definition}\label{UniDep}(\cite[Definition 3.7.]{bakker2018kodaira}) The uniform depth of cusps of $X$ is the largest $d$ satisfying the following properties, setting $v_i=t_i/d$:
\begin{enumerate}
     \item for every $i$, $v_i\le u_i$ (this gives that $H_i(v_i)$ injects into $X$).     
     \item all $H_i(v_i)$ are disjoint.  
\end{enumerate}

  Thanks to Parker's generalization of Shimizu's lemma \cite[Proposition 2.4.]{parker1998volumes}, the uniform depth of cups is at least $2$ for torsion free.  
\end{definition}

\begin{Remark}
By passing to a cover of $X$, i.e., passing to a subgroup of $\Gamma$, the shortest vertical translations of parabolic subgroups can not decrease and the heights can not increase. Therefore, the canonical depth can not decrease by passing to a cover of $X$.       
\end{Remark}

\begin{definition}(see \cite[Section 2]{yeung1994betti}) \label{Tow} A cofinal normal tower of $X$ is 
a sequence  $\displaystyle \{X_i\}_{i=1}^{\infty}$ of \'etale Galois coverings of $X=X_1$ given by a sequence of lattices $\displaystyle \{\Gamma_i\}_{i=1}^{\infty}$ such that for each subsequence $\displaystyle \{\Gamma_j\}_{j=1}^{\infty},$ one has  
$\bigcap_{j=1}^{\infty} \Gamma_{j}= \{1\}.$ 
\end{definition}
Similar to \cite[Lemma 1.2.2.]{yeung1994betti} and \cite[Proposition 2.1.]{hwang2006uniform}, we can prove that the uniform depth of cusps tends to infinity in cofinal normal towers:    
\begin{Lemma} \label{tower}
Let $\{X_{i}\}_{i=1}^{\infty}$ be a cofinal normal tower of $X$ and $d(X_i)$ be the uniform depth of cusps of $X_i.$ Then, 
$$\lim_{i \to \infty} d(X_i) = \infty.$$
\end{Lemma}

\begin{proof} Let $p_i: X_i \to X $ be the covering map. Fix a cusp $c_1$ of $X.$ Since $X_i$ is a normal cover of $X,$ the group 
$\Gamma$ acts transitively on all cusps of $X_i$ that get sent to $c_1,$ and therefore the uniform depth can be computed by taking only one cusp over each cusp of $X.$ Let $c_i$ be a cusp of $X_i$ such that $p_i(c_i)=c_1.$ Let $u(c_i)$ be the height of cusp $c_i,$ and $t(c_i)$ be the shortest vertical translation in the stabilizer of $c_i.$ Since $u(c_i)\le u(c_1), $ it is sufficient to show that $t(c_i)$ tends to infinity. Note that as $\Gamma
$ is discrete, for every $t_i$ there only exists finitely many elements in $\Gamma$ with the length smaller than or equal to $t_i.$ On the other hand, for any subsequence $\displaystyle \{\Gamma_j\}_{j=1}^{\infty},$ we have   
$\bigcap_{j=1}^{\infty} \Gamma_{j}= \{1\}.$ Putting these together, we conclude that the length of the shortest element except $1$ in $\Gamma_i$ tends to infinity.    

\end{proof}
\begin{Remark}\label{modularcurve} In general, the uniform depth of cusps tends to infinity in most of the towers of covering, even in the case that ramification index of some cusps are fixed in the tower. More precisely, for  towers that the traces of hyperbolic elements tends to infinity, the uniform depth tends to infinity. This is proved in \cite[section 3]{memarian2023volumes}, and here we examine the special case of modular curves $X_0(p),$ where $p$ is a prime number or $1$. 
Consider the lattice 
\begin{align*}
    \Gamma_{0}(p)= \Big\{ \begin{pmatrix}
a & b \\
c & d 
\end{pmatrix}\in \operatorname{PSL}_2(\Z) \mid
\begin{pmatrix}
a & b \\
c & d 
\end{pmatrix} \cong
\begin{pmatrix}
* & * \\
0 & * 
\end{pmatrix}
 mod \ p \Big\}. 
\end{align*}

The group $\operatorname{PSL}_2(\R)$ acts on $\mathbb{H}$ with the linear fractional translation and the action naturally extends to $\mathbb{H}^*=\mathbb{H}\cup \Q \cup \{\infty\}.$ 
Let $$X_0(p):=\Gamma_{0}(p) \backslash \mathbb{H}^*,$$
and consider the covering $X_{0}(p)\to X_{0}(1).$ It is well-know that for prime $p,$ the curve $X_0(p)$ has only two inequivalent cusps, see \cite[pg 26]{shimura1971introduction}, and $0$ and $\infty$ represent these two inequivalent cusps of $X_0(p).$ We will denote the stabilizer of $0$ and $\infty$ in $\Gamma_{0}(p)$ by $\operatorname{stab}_{\Gamma_{0}(p)}(0)$ and $\operatorname{stab}_{\Gamma_{0}(p)}(\infty):$ 
$$
\operatorname{stab}_{\Gamma_{0}(p)}(0)=\Big\{ \begin{pmatrix}
1 & 0 \\
pk & 1 
\end{pmatrix} \mid
k \in \Z  \Big\},
\
\operatorname{stab}_{\Gamma_{0}(p)}(\infty)=\Big\{ \begin{pmatrix}
1 & b \\
0 & 1 
\end{pmatrix} \mid b\in \Z
  \Big\}.$$

Therefore, the ramification indices of $0$ and $\infty$ over $X_0(1)$ are $p=[\operatorname{stab}_{\Gamma_{0}(p)}(0):\operatorname{stab}_{\Gamma_{0}(1)}(0)]$ and $1=[\operatorname{stab}_{\Gamma_{0}(p)}(\infty):\operatorname{stab}_{\Gamma_{0}(1)}(\infty)],$ respectively.

Consider $g=\begin{pmatrix}
a & b \\
c & d 
\end{pmatrix}\in \Gamma_{0}(p)$ which is not in $\operatorname{stab}_{\Gamma_{0}(p)}(\infty).$ It is straightforward to check that for every $z\in \mathbb{H},$ we have $\operatorname{Im}(g(z))\cdot \operatorname{Im}(z)\le \frac{1}{|c|^2}.$ Since $|c|\ge p,$ this implies that the horoball $\operatorname{stab}_{\Gamma_0(p)}(\infty)\backslash \{\operatorname{Im}(z)> \frac{1}{p}\}$ injects to $\Gamma_0(p) \backslash \mathbb{H}^1$ and therefore the height of $\infty$ is at most $\frac{1}{p}.$ Since the length of shortest vertical translation around $\infty$ is $1,$ we get that the depth of $\infty$ is at least $p.$ To find the depth of $0,$ consider the element $h=\begin{pmatrix}
0 & 1 \\
-1 & 0 
\end{pmatrix}\in \operatorname{PSL}_2(\Z)$ which sends $0$ to $\infty.$ Let 
$\operatorname{stab}'_{\Gamma_{0}(p)}(0)$ and $\Gamma'_0(p)$ be the conjugation of $\operatorname{stab}_{\Gamma_{0}(p)}(0)$ and $\Gamma_0(p)$ by $h,$ respectively. Then, we have
\begin{align*}
    \Gamma_{0}'(p)= \Big\{ \begin{pmatrix}
a' & b' \\
c' & d' 
\end{pmatrix}\in \operatorname{PSL}_2(\Z) \mid
\begin{pmatrix}
a' & b' \\
c' & d' 
\end{pmatrix} \cong
\begin{pmatrix}
* & 0 \\
* & * 
\end{pmatrix}
 mod \ p \Big\}, 
\end{align*}
$$\operatorname{stab}'_{\Gamma_{0}(p)}(0)=\Big\{  \begin{pmatrix}
 1 & pk \\
0 & 1 
\end{pmatrix} \mid
k \in \Z  \Big\}.$$
Similarly, for every $g'=\begin{pmatrix}
a' & b' \\
c' & d' 
\end{pmatrix}\in \Gamma'_{0}(p)$ which is not in $\operatorname{stab}'_{\Gamma_{0}(p)}(0),$ and every $z\in \mathbb{H},$ we have that $\operatorname{Im}(g'(z))\cdot \operatorname{Im}(z)\le \frac{1}{|c'|^2}.$  Therefore, the horoball $\operatorname{stab}'_{\Gamma_0(p)}(\infty)\backslash \{\operatorname{Im}(z)> 1\}$ injects to $\Gamma'_0(p) \backslash \mathbb{H}^1.$  This implies that the height of $0$ is at most $1$ and because the length of shortest vertical translation is $p,$ we get that the depth of $0$ is at least $p.$   
Since the set $\{\operatorname{Im}(z)> \frac{p}{d}\}$ is disjoint from 
the set  $\{\operatorname{Im}(z)< \frac{p}{d}\},$ for every $d\le \sqrt{p},$ we can conclude that the uniform depth of cusps of $X_0(p)$ is $\sqrt{p}.$



\end{Remark}

\subsection{Toroidal Compactification}

 The variety $X$ has a unique toroidal compactification in the sense of \cite{Cmpftification, MOK}.
This torodial compactification is a smooth projective variety and each connected component of the boundary divisor $D$ is an \'etale quotient of an abelian variety whose conormal bundle is ample. Moreover, if $\Gamma$ has only unipotent parabolic subgroups(in particular when $\Gamma$ is neat), then $D$ is disjoint union of abelian varieties with ample conormal bundle. Note that $\Gamma$ always has a finite-index neat subgroup by \cite{Cmpftification} in the arithmetic case and by \cite{hummel1998rank} in general.     

We describe the toroidal compactification $\Xb$ of $X$ locally around the cusp at $\infty$. Consider the holomorphic map 
\begin{align*}
    \Sn \ &\longrightarrow \C^{n-1}\times \Delta^{*} \\
    (\zeta, z) &\longrightarrow (\zeta, e^{2\pi i z/t_{\infty}}),
\end{align*}
which is surjective and invariant under the action of $W_{\infty}:=V_{\infty}\cap \Gamma$. This map identifies the partial quotient $W_{\infty}\backslash \Sn$ by its image. Since $\Lambda_{\infty}=W_{\infty} \backslash \Gamma_{\infty}$ is identified to a discrete unitary transformation of $\C^{n-1}$, the boundary divisor in the compactification of  $\Gamma
_{\infty}\backslash \Sn$ is given by $\Lambda_{\infty} \backslash \C^{n-1}\times 0$. Note that not having an orbifold point on the compactification is equivalent to not having a torsion element in $\Lambda_{\infty}$.

Therefore, assuming that $\Xb$ does not have an orbifold point gives that the residual quotient of $\C^{n-1}\times \Delta$ by $\Lambda_{\infty}$ is locally \'etale quotient and the coordinates $\zeta, q=e^{2\pi i z/t_{\infty}}$ extends to coordinates of $\Xb$ around the boundary such that the boundary divisor in this neighborhood is cut out by $q=0$. We will denote the connected component of $D$ by $D_i$ which is an \'etale quotient of an abelian variety with ample conormal bundle(\cite{MOK}). In the case that 
the parabolic subgroups of $\Gamma$ are unipotent (e.g. $\Gamma$ is neat), the boundary  is a disjoint union of abelian varieties.
\section{Positivity of $O_{\Proj(\Omega^1_X)}(1)$}\label{Sec2}

Let $E$ be a vector bundle of rank $n$ on an algebraic variety or complex manifold $Y$ and let 
$$\pi:\Proj(E)\to Y$$ 
be the projective bundle of lines in the dual bundle $E^*$.
The projective bundle $\Proj(E)$ carries a tautological quotient bundle line bundle $O_{\mathbb{P(E)}}(1)$ whose dual $O_{\mathbb{P(E)}}(-1)$ is naturally a subbundle of $\pi^{*}E$:    
\begin{align*}
 0\longrightarrow O_{\Proj(E)}(-1)\longrightarrow \pi^{*}E.   
\end{align*}
In other words, a point $(y,[L])$ of $\Proj(E)$ is determined by a point $y\in Y$ together with $[L]$ in $\Proj^n= \Proj\bi(E^{*}(y)\bi)$. The fiber $O_{\Proj(E)}(-1)_{(y,[L])}$ is the subvector space $L$ and the fiber $O_{\Proj(E)}(1)_{(y,[L])}$ is the one dimensional quotient of $E$ corresponding to $L$. A vector bundle $E$ is called ample if $O_{\Proj(E)}(1)$ is ample.  

Let $X=\Gamma \backslash\mathbb{B}^n$ be a complex ball quotient by a torsion-free lattice $\Gamma \subset \operatorname{Aut}(\mathbb{B}^n)$ and $\hat{h}$ be the hermitian metric induced by the Bergman metric $h$ on $O_{\Proj(\Omega^1_{X})}(1)$.

In this section, we prove two properties of $\hat{h}$ in Proposition \ref{Form Inequlity} and Proposition \ref{Good metric}. These properties will be used to construct a singular hermitian metric on the line bundle $O(1)$ over $\Proj(\lgcot)$ to prove Theorem \ref{SlopLog}. 

Let  $\hat{h}^*$ be the dual metric of $\hat{h}$ on 
$O_{\Proj(\Omega^1_{X})}(-1)$. Suppose $U$ is an open subset of $\Proj(\Omega^{1}_{X})$ such that for every $x \in U$, $\xi_{1, x}\neq 0$.  After replacing $\xi_{i}$ by $\frac{\xi_{i}}{\xi_{1}}$ and $w$ by $\frac{w}{\xi_i}$, we can take a local section $\sigma=\dis\sum_{i=1}^{n-1}\xi_{i}e_{i}+we_{n}$ of $O_{\Proj(\Omega^1_X)}(-1)$ on $U$. The first Chern form of $O_{\Proj(\Omega^1_X)}(-1)$ on $U$ is represented by 
\begin{align*}
c_{1}(O_{\Proj(\Omega^1_X)}(-1),\hat{h}^*)
&= 
\frac{i}{2\pi}\delb\del \log||\sigma||_{\hat{h}^*}^{2} \\
&= 
\frac{i}{2\pi}\delb\del \log\B(\sum_{i=1}^{n}\sum_{j=1}^{n}\xi_{i}\bar{\xi}_{j}h(e_i,e_j)\B)\\
&=\frac{i}{2\pi}
(-2\delb\del \log u
+
\delb\del \log v),
\end{align*}
where \begin{align} \label{functionV}
v=\dis\sum_{i=1}^{n-1}\sum_{j=1}^{n-1}\bi( (u\delta_{ij}+\bar{\zeta_i} \zeta_j)\xi_i\bar{\xi}_j
+
\frac{1}{2i}\xi_j \bar{\xi}_n \bar{\zeta_{j}}
-
\frac{1}{2i} \bar{\xi}_j \xi_n \zeta_{j}
\bi) 
+
\frac{1}{4}|\xi_n|^2.\end{align}
As $O_{\Proj(\Omega^{1}_{X})}(-1)$ is the dual of $O_{\Proj(\Omega^{1}_{X})}(1)$, the first Chern form $c_{1}(O_{\Proj(\Omega^1_X)}(1),\hat{h})$ on $U$ is given by
\begin{align}\label{ChernForm}
c_{1}\bi(O_{\Proj(\Omega^1_X)}(1),\hat{h}\bi)
=
\frac{i}{2\pi}(2\delb\del \log u
-
i\delb\del \log v).   
\end{align}

\begin{Lemma} \label{Homogenious}
For every $q \in \Omega_{X}^{1}$ there exists a point $p$ in the orbit $I_{\infty} \cdot q$ such that in $(\zeta, z, \xi, w)$-coordinates $p$ is given by 
$$\zeta_{i,p}=\xi_{i,p}=0,$$
for $1\le i\le n-2$, $\zeta_{n-1,p}=0$ and $\xi_{n-1,p}$ being a positive real number.
\end{Lemma}

\begin{proof} Consider the action of $(A,\tau,t)\in \operatorname{U}(n-1)\times \C^{n-1} \times \R$ described in \ref{ActonCot} and take $(A,\tau,t)= (I,-\zeta_{q}, 0)$:
$$(I,-\zeta, 0).(\zeta_q,z_q,\xi_q, w_q)
=
(0,z',\xi_q +i|\zeta_q|^2 +2i(\xi_q, -\zeta_q), w_q)
$$

There exist $B \in \operatorname{U}(n-1)$ such that $$B \bi(\xi + 2i(\xi, -\zeta)\bi) = (0,...,0,|\xi_q + 2i(\xi_q, -\zeta_q)|).$$
The element $(B,0,0) \in U(n-1)\times \C^{n-1} \times \R$ sends $(0,z',\xi_q +i|\zeta_q|^2 +2i(\xi_q, -\zeta_q), w_q)$ to a point $p$ with the desired properties.
\end{proof}

\begin{Proposition}\label{Form Inequlity} Let $X=\Gamma \backslash \mathbb{B}^n$ be a torsion-free ball quotient and $\hat{h}$ be the hermitian metric on $O_{\Proj(\Omega_{X}^{1})}(1)$ induced by the Bergman metric on $\mathbb{B}^n$. Then, 
\label{Positive form}
\begin{enumerate}[label=(\roman*)]
    \item the first Chern form $c_{1}\bi(O_{\Proj(\Omega_{X}^{1})}(1),\hat{h}\bi)$ is a Kähler form on $\Proj(\Omega_{X}^{1})$; \label{Kähler form}
    \item $$ c_{1}\bi(O_{\Proj(\Omega_{X}^{1})}(1),\hat{h}\bi)  \ge \frac{1}{4\pi}\pi^{*}(w_{X}),$$ \label{Form comparsion}
     where $w_{X}$ is the Kähler form of the Bergman metric on $X$.
\end{enumerate}
\end{Proposition}
\begin{proof}
First we prove part (ii) and then conclude part(i) from the inequality appears in the proof of (ii).    

(ii) Since the Bergman metric and $u$ is $I_{\infty}$-invariant, its enough to check the inequality on $I_{\infty}$-orbit. Consider an open set $U\subset \Proj(\Omega^1_{X})$ such that $w\neq 0$. Replacing $\xi_i$ by $\frac{\xi_i}{w}$ and $w$ by $1$, we can work on the affine coordinates on $\Proj(\Omega^1_{X})_{|U}$. Thanks to Lemma \ref{Homogenious}, we can move every point in $U$ by an element of $I_{\infty}$ to point $p$ such that $$\zeta_{1,p}=\zeta_{2,p}=...=\zeta_{n-1, p}=0,$$ $$\xi_{1,p}=\xi_{2,p}=...=\xi_{n-2,p}=0,$$ $w=1$ and $\xi_{n-1,p}$ is a real number. Note that the function $v$ in  \ref{functionV} on $U$ is 

$$v= (\dis\sum_{i=1}^{n-1}\sum_{j=1}^{n-1}\bi( (u\delta_{ij}+\bar{\zeta_i} \zeta_j)\xi_i\bar{\xi}_j
+
\frac{1}{2i}\xi_j  \bar{\zeta_{j}}
-
\frac{1}{2i} \bar{\xi}_j \zeta_{j}
\bi) 
+
\frac{1}{4}.$$ 
Set $L=O_{\Proj(\Omega^1_{X})}(1).$ The first Chern form at $p$ is

$$c_{1}(L,\hat{h})(p)
=
\frac{i}{2\pi}\delb\del\log(v)(p)
=
\frac{i}{2\pi} (v(p))^{-2} \bi(v(p)\cdot \delb\del v(p)-(\delb v\wedge \del v)(p) \bi).$$
To find explicit formula note that 
\begin{align*}
   r &:=v(p)=u\xi_{n-1}^2+\frac{1}{4},\\
   \del v(p)&=u\xi_{n-1}d\xi_{n-1}+\frac{1}{2i}\xi_{n-1}^2dz-\frac{1}{2i}\xi_{n-1}d\zeta_{n-1},
\end{align*}
and 
\begin{align*}
    \delb\del v(p)
=& 
\dis\sum_{i=1}^{n-1}
\bi( 
-
\xi_{n-1}^2d\bar{\zeta}_{i} \wedge d\zeta_{i}
+
u d\bar{\xi}_{i}\wedge d\xi_{i}
+
\frac{1}{2i}d\bar{\zeta_{i}} \wedge d\xi_{i}
-
\frac{1}{2i}  d\bar{\xi}_{i} \wedge d\zeta_{i} 
\bi)\\
&+\xi_{n-1}^2d\bar{\zeta}_{n-1} \wedge d\zeta_{n-1}
-
\frac{1}{2i}\xi_{n-1}d\bar{z} \wedge d\xi_{n-1}
+
\frac{1}{2i}\xi_{n-1} d\bar{\xi}_{n-1} \wedge dz,
\end{align*}
giving that
\begin{align*}
c_{1}(L,\hat{h})(p)
=
\frac{i}{2\pi} r^{-2}\cdot \B (&
\dis\sum_{i=1}^{n-2}
\bi( 
r u d\bar{\xi}_{i}\wedge d\xi_{i}
-
r\xi_{n-1}^2d\bar{\zeta}_{i} \wedge d\zeta_{i}
+
\frac{1}{2i}r d\bar{\zeta_{i}} \wedge d\xi_{i}
-
\frac{1}{2i}r  d\bar{\xi}_{i} \wedge d\zeta_{i} 
\bi)\\
&+\frac{1}{4}\cdot \bi(
u d\bar{\xi}_{n-1}\wedge d\xi_{n-1}
+
\frac{1}{2i}d\bar{\zeta}_{n-1} \wedge d\xi_{n-1}
-
\frac{1}{2i}  d\bar{\xi}_{n-1} \wedge d\zeta_{n-1}\\
&-
\frac{1}{2i}\xi_{n-1}d\bar{z} \wedge d\xi_{n-1}
+
\frac{1}{2i}\xi_{n-1} d\bar{\xi}_{n-1} \wedge dz
-
\xi^4_{n-1}d\bar{z}\wedge dz \\
&+
\xi_{n-1}^3d\bar{z} \wedge d\zeta_{n-1} 
+
\xi_{n-1}^3d\bar{\zeta}_{n-1} \wedge dz
-
\xi_{n-1}^2d\bar{\zeta}_{n-1}\wedge d\zeta_{n-1}\bi) 
\B).
\end{align*}
Now, as $\omega_{X}=2i\delb \del \log(u)$, using equation \ref{ChernForm}, we can write 
\begin{align*}
c_{1}(L,\hat{h})(p)-\frac{1}{4\pi}\pi^{*}(w_{X})(p)
=
\frac{i}{2\pi}\bi(\delb\del \log (u) (p)
-
\delb\del\log(v)(p)\bi),
\end{align*}
and therefore denoting the form $c_{1}(L,\hat{h})-\frac{1}{4\pi}\pi^{*}(w_{X})$ by $\eta$, we obtain 
\begin{align*}
\eta(p)
=
&\frac{i}{2\pi}\cdot \dis\sum_{i=1}^{n-2}
\bi( u r ^{-1} d\xi_{i} \wedge d\bar{\xi_i}
+
(u^{-1} -r ^{-1}\xi_{n-1}^2)d\zeta_{i} \wedge  d\bar{\zeta}_{i}
+
\frac{1}{2i}r^{-1} d\xi_{i} \wedge  d\bar{\zeta_{i}}
-
\frac{1}{2i} r ^{-1}  d\zeta_{i} \wedge  d\bar{\xi}_{i}
\bi )\\
&+
\frac{i}{8} r ^{-2}u^{-2}
\B (
(8ur^{2}-2u^2\xi_{n-1}^2) d\zeta_{n-1} \wedge d\bar{\zeta}_{n-1}
+
2u^3d\xi_{n-1} \wedge d\bar{\xi}_{n-1}
-
iu^2  d\xi_{n-1} \wedge d\bar{\zeta}_{n-1}\\
& \hspace{2.8cm}+
iu^2   d\zeta_{n-1} \wedge d\bar{\xi}_{n-1}
+
iu^2\xi_{n-1} d\xi_{n-1} \wedge  d\bar{z}
-
iu^2\xi_{n-1}   dz \wedge d\bar{\xi}_{n-1}\\
& \hspace{2.8cm}+
(2r^2
-
2u^2\xi^4_{n-1}) dz \wedge  d\bar{z}
+
2u^2 \xi_{n-1}^3  d\zeta_{n-1} \wedge  d\bar{z}
+
2u^2\xi_{n-1}^3  dz \wedge d\bar{\zeta}_{n-1}
\B)
\end{align*}
To see that $\eta(p)$ is semi-positive, we choose the local frame $\beta$ for $L$ over $U$ as follows:  
$$\beta_{2i-1}=\frac{\del}{\del \xi_{i}},\ \beta_{2i}=\frac{\del}{\del \zeta_{i}},$$
for $1\le i\le {n-1}$ and $\beta_{2n-1}=\frac{\del}{\del z}.$ 
In this frame, 
$$[\eta(p)]_{\beta}= \frac{i}{2\pi}
\begin{bmatrix}
A_1& 0 & ...&0&0\\
0&A_2& ... & 0&0\\
...&...& ... & ... &...\\
0 & 0&...&A_{n-1}&0\\
0 & 0&...&0&B
\end{bmatrix},$$
where $A_{i}=r^{-1}
\begin{bmatrix}
u&\frac{1}{2i}\\
\\
-\frac{1}{2i}&\frac{1}{4u}
\end{bmatrix}$
and 
$B=
\frac{1}{8}r^{-2}u^{-2}\begin{bmatrix}
2u^3 & i u^2 & -i u ^2\xi_{n-1}\\
\\
-i u ^ 2 & 8u r ^{2}-2u^2\xi_{n-1}^2 & 2u^2\xi^3\\
\\
i u ^2\xi_{n-1} & 2u^2\xi^3 & 2 r ^2
-
2 u ^ 2 \xi ^ 4 _{n-1}

\end{bmatrix}.$

By computation we can see that $\operatorname{det}(A_{i})=0$, $\operatorname{tr}(A_{i})>0$ and determinate of all upper left sub-matrices of $B$ are semi-positive. Hence, 
$$c_{1}\bi(L,\hat{h}\bi)\ge\frac{1}{4\pi}w_{X}.$$ 

(i) Since $c_{1}\bi(L,\hat{h}\bi)$ is a closed form, we only need to show that it is positive. As $\omega_{X}$ is zero only on the vertical directions and neither $A_i$s nor $B$ is zero on the vertical directions, $c_{1}(L, \hat{h})$ is a positive $(1,1)$-form.

\end{proof}

In fact, if $\Gamma \subset \operatorname{PU}(n,1)$ is a cocompact lattice, Proposition \ref{Positive form} implies that $\Gamma \backslash\mathbb{B}^n$ has ample cotangent bundle. This is well-know and the difficulty is, when $\Gamma$ is not cocompact.

For the non-compact case, we will construct a hermitian metric on $O(1)$ over  $\Proj(\Omega^{1}_{X})$ which extends as singular hermitian metric to $O(1)$ over $\Proj(\lgcot).$

To prove Theorem \ref{SlopLog}, we will construct a singular hermitian metric whose curvature current is represented by a form. To this end, we prove the following Proposition which is inspired by \cite[Lemma 6.18]{griffiths1973nevanlinna} and \cite[Proposition 5.16]{kollar1985subadditivity}:

\begin{Proposition}\label{Current}
Suppose $\Phi(\zeta, |q|):\C^{n-1}\times \C\to (0,\infty]$ is a function satisfying the following conditions:
\begin{enumerate} [label=(\roman*)]
    \item $\del \log\bi(\Phi(\zeta, |q|)\bi)$ and $\del\delb \log\bi(\Phi(\zeta, |q|)\bi)$ are locally integrable on a neighborhood of the divisor $q=0$;
    
    \item $\dis \lim_{q\to 0}\frac{\log(\Phi(\zeta, |q|))}{\log|q|}=0,$ when $|\zeta|$ is bounded. 
\end{enumerate}

Then, the current $\Big[\del\delb \log\bi(\Phi(\zeta, |q|)\bi)\Big]$ is represented by the form $\del\delb \log\bi(\Phi(\zeta, |q|)\bi).$  

\end{Proposition}


\begin{proof}

Let $U$ be an open set where the divisor $D$ is given by $q=0$. For a compactly supported $(n-1,n-1)-$form  $f$ on $U$,

$$
\dis\int_{X}\log\bi(\Phi(\zeta, |q|)\bi)\del\delb f
=
\dis\lim_{\epsilon\to 0}\dis\int_{X_{\epsilon}}\log\bi(\Phi(\zeta, |q|)\bi)\del\delb f,
$$
where $X_\epsilon=\{(\zeta,q) \in \C \ | \ |q|>\epsilon\}$.
Applying Stokes' theorem to the closed form $d\bi(\log|\Phi(\zeta, |q|)\del f\bi)$, we get:

\begin{align} \label{Stokes1}
\dis\int_{X_{\epsilon}}\log(\Phi(\zeta, |q|))\delb\del f
=
-
\dis\int_{X_{\epsilon}} \delb\log\bi(\Phi(\zeta, |q|)\bi)\wedge \del f
+
\dis\int_{S_{\epsilon}}\log\bi(\Phi(\zeta, |q|)\bi)\del f,
\end{align}
where $S_\epsilon= \{(\zeta,q) \in \C^n \ | \ |q|=\epsilon\}$ is oriented with its normal in the direction of decreasing $|q|$.
Since $S_\epsilon= C_\epsilon \times \C^{n-1}$ where $C_\epsilon$ is the circle $\{ q\in \C \ | \ |q|=\epsilon\}$,  Fubini's theorem gives that 
\begin{align}
\dis\int_{S_{\epsilon}}\log\bi(\Phi(\zeta, |q|)\bi)\del f
=
&\dis\int_{\C^{n-1}}\B(\int_{C_{\epsilon}}\log(\Phi\bi(\zeta, |q|\bi)\tildeb{f}(\zeta, q)dq\B) \ d\zeta \wedge d\bar{\zeta} \nonumber \\ \label{FirstZero}
= 
&\dis\int_{D}\B( \log\bi(\Phi(\zeta,\epsilon)\bi) \int_{C_{\epsilon}}\tildeb{f}  (\zeta, q)dq\B) \ d\zeta \wedge d\bar{\zeta},
\end{align}
where $d\zeta\wedge d\bar{\zeta} = d\zeta_1\wedge d\bar{\zeta_1} \wedge ... d\zeta_{n-1}\wedge d\bar{\zeta}_{n-1}$,  $ \tildeb{f}(\zeta,q)dq\wedge d\zeta \wedge d\bar{\zeta} $ is $dq\wedge d\zeta \wedge d\bar{\zeta}$ part of $\del f$ and $D$ is a compact set such that $D \times C_\epsilon$ contains the support of $\tildeb{f}(\zeta,q)$.
Since $\dis \lim_{q\to 0}\frac{\log(\Phi(\zeta, |q|))}{\log|q|}=0$, there exists $\epsilon'> 0$ such that $|\log\bi(\Phi(\zeta, \epsilon)\bi)|<\epsilon' |\log(\epsilon)|$. 
Furthermore, as $\tildeb{f}(\zeta,q)$ is continuous and $C_\epsilon$ is compact, there exists a continuous function $M(\zeta)$ such that
$\dis\left|\int_{C_{\epsilon}}\tildeb{f}  (\zeta, q)dq \ \right|< M(\zeta)\cdot \epsilon$. It follows by compactness of $D$ that
\begin{align} \label{FirstEstimation}
    \dis \left| \int_{D}\B( \log\bi(\Phi(\zeta, \epsilon)\bi) \int_{C_{\epsilon}}\tildeb{f}(\zeta, q)dq\B) \ d\zeta \wedge d\bar{\zeta} \right|
\le R \cdot \epsilon \log(\epsilon),
\end{align}
where $R$ is a constant. As such, \ref{FirstZero} and \ref{FirstEstimation} imply that $\dis\lim_{\epsilon \to 0 }\dis\int_{S_{\epsilon}}\log(\Phi(\zeta, |q|))\del f
=0$ and by \ref{Stokes1} we obtain that
\begin{align} \label{Cleaned First Term}
  \int_{X}\log\bi(\phi(\zeta, |q|)\bi)\del\delb f
=
\int_{X} \delb\log\bi(\phi(\zeta, |q|)\bi)\wedge \del f.
\end{align}

On the other hand, applying Stokes' theorem to the closed form $d\B(f \wedge \delb\log\bi(\phi(\zeta, |q|)\bi)\B)$ yields that:

\begin{align} \label{Second Term}
\dis\int_{X_\epsilon}f\wedge \del\delb\log\bi(\Phi(\zeta, |q|)\bi)
=
-
\dis\int_{X_{\epsilon}} \del f \wedge \delb\log\bi(\Phi(\zeta, |q|)\bi)
+
\dis\int_{S_{\epsilon}} f\wedge \delb\log\bi(\Phi(\zeta, |q|)\bi).
\end{align}

It will be shown that $\dis\lim_{\epsilon\to 0}\int_{S_{\epsilon}} f\wedge \delb\log\bi(\Phi(\zeta, |q|)\bi)=0$. Note that

\begin{align} \label{SecondVanishing}
    \dis\int_{S_{\epsilon}} f\wedge \delb\log\bi(\Phi(\zeta, |q|)\bi)
    &= 
     \int_{r=\epsilon} f \wedge
     \frac{1}{\Phi(\zeta,r)} \cdot\frac{\del \Phi(\zeta, r)}{\del \bar{q}}
     +
     \int_{|q|=\epsilon} f
    \wedge
    \sum_{i=1}^{n-1} \frac{\del \log\bi(\Phi(\zeta, |q|)\bi)}{\del \bar{\zeta_i}},
\end{align}
where $r=|q|$. Applying L'Hopital's rule to $\dis\lim_{r\to 0}\dfrac{\log(\Phi(\zeta,r))}{log(r)}=0$  yields 
\begin{align} \label{Hopital}
\lim_{r\to 0}\frac{r}{\Phi(\zeta,r)}\cdot \frac{\del \Phi(\zeta,r)}{\del r}=0.    
\end{align}

As $\frac{\del r}{\del \bar{q}}= \frac{1}{2}\B(\frac{q}{\bar{q}}\B)^\frac{1}{2}$
the chain rule gives
$$\dis \frac{\del\Phi(\zeta,|q|) }{\del \bar{q}}
=
\frac{1}{2}\B(\frac{q}{\bar{q}}\B)^\frac{1}{2}
\cdot
\frac{\del \Phi(\zeta,r)}{\del r}.
$$
Since $f$ is compactly supported, there exists a constant $M'$ such that

\begin{align} \label{Lelong}
    \dis\left |\int_{r=\epsilon} f \wedge
     \frac{1}{\Phi(\zeta,r)} \cdot\frac{\del \Phi(\zeta, |q|)}{\del \bar{q}} \right|
     \le M'\epsilon \cdot \max_{r=\epsilon}{\left| \frac{1}{\Phi(\zeta,r)}\cdot \frac{\del \Phi(\zeta,r)}{\del r} \right|},
\end{align}
which tends to $0$ as $\epsilon$ tends to 0 using \ref{Hopital}. 
The other term in \ref{SecondVanishing} can be written as

\begin{align}
 \int_{|q|=\epsilon} f
    \wedge
    \sum_{i=1}^{n-1} \frac{\del \log\bi(\Phi(\zeta, |q|)\bi)}{\del \bar{\zeta_i}}
    &= \int_{\C^n}\sum_{i=1}^{n-1} \frac{\del \log\bi(\Phi(\zeta, \epsilon)\bi)}{\del \bar{\zeta_i}}\int_{|q|=\epsilon}f \nonumber \\ \label{StokesOfZeta}
    &= \int_{\C^n} \log\bi(\Phi(\zeta,\epsilon)\bi) \ \delb\B 
    (\int_{|q|=\epsilon}f\B),
\end{align}
by Stokes' theorem. As  $f$ is compactly supported there exists a constant $R'$ such that

\begin{align} \label{Zeta}
\left|\int_{\C^n} \log\bi(\Phi(\zeta,\epsilon)\bi) \ \delb\B
    (\int_{|q|=\epsilon}f\B)\right|
\le
R' \epsilon' |\log(\epsilon)| \epsilon 
\end{align}
tending to $0$ as $\epsilon$ tends to $0$. Therefore, combining \ref{SecondVanishing}, \ref{Lelong}, \ref{StokesOfZeta} and \ref{Zeta} gives 
$\dis \lim_{\epsilon \to 0}\int_{S_{\epsilon}} f\wedge \delb\log\bi(\Phi(\zeta, |q|)\bi)=0$ and by \ref{Second Term} we get 
\begin{align} \label{Cleaned Second Term}
\dis\int_{X}f\wedge \del\delb\log\bi(\Phi(\zeta, |q|)\bi)
=
-
\dis\int_{X} \del f \wedge \delb\log\bi(\Phi(\zeta, |q|)\bi).
\end{align}

The claim follows from \ref{Cleaned First Term} and \ref{Cleaned Second Term}.
\end{proof}

\begin{Proposition} \label{Good metric}
Let $s$ be a local section of $O_{\Proj(\Omega^1_{X}(\log D))}(1).$ Then, the current $[\del\delb \log(\hat{h}(s))]$ is represented by the form $\del \delb \log(\hat{h}(s)).$

\end{Proposition}

\begin{proof} Let $D'$ be the pull back of $D=\Xb\backslash X$ to $\Omega^1_{\Xb}(\log D)$. Take an arbitrary point $a =(\zeta_{1},..., \zeta_{n-1},q=0;[\xi_{1},...,\xi_{n}])\in D'$. At least one of $\xi_1,\xi_2,...,\xi_n$ is non-zero.
Let $\xi_{k} \neq 0,$ for some $k\in \{1,2,...,n\}.$ Suppose $U$ is an open subset of $\Proj(\Omega^1_{X}(\log D))$ such that for every $u \in U$, $\xi_{k, u}\neq 0$ and $D' \cap U =(q=0)$. Replacing $\xi_{i}$ by $\frac{\xi_{i}}{\xi_{k}}$, 
$\eta=\dis\sum_{i=1}^{n-1}\xi_i\frac{\del}{\del \zeta_{i}}+\xi_{n}q\frac{\del}{\del q}$ is a local section of $O_{\Proj(\Omega^1_{X}(\log D))}(-1)$ on $U$. Note that on $U$ we have $h(s)=\dfrac{1}{||\eta||^2_{\hat{h}^*}},$ and therefore it is enough to check the conditions of Proposition \ref{Current} for $\log(||\eta||^2_{\hat{h}^*}).$  We write this function in terms of the chosen coordinate. First, we have that 

\begin{align*} 
||\eta||^2_{\hat{h}^*}&= 
u^{-2}.\B(\sum_{i=1}^{n-1}\sum_{j=1}^{n-1} \bi( (u\delta_{ij}+\bar{\zeta_i} \zeta_j)\xi_i\bar{\xi}_j
+
\frac{t_{\infty}}{4\pi}\xi_j \bar{\xi}_n \bar{\zeta_{j}}
+
\frac{t_{\infty}}{4\pi} \bar{\xi}_j \xi_n \zeta_{j}
\bi) 
+
\frac{t^2_{\infty}}{16\pi}|\xi_n|^2
\B)\\
&= u^{-1}\cdot \sum_{i=1}^{n-1}|\xi_i|^2+u^{-2}\cdot \B(\sum_{i=1}^{n-1}\sum_{j=1}^{n-1}\bi( \bar{\zeta_i} \zeta_j\xi_i\bar{\xi}_j
+
\frac{t_{\infty}}{4\pi}\xi_j \bar{\xi}_n \bar{\zeta_{j}}
+
\frac{t_{\infty}}{4\pi} \bar{\xi}_j \xi_n \zeta_{j}
\bi) 
+
\frac{t^2_{\infty}}{16\pi}|\xi_n|^2
\B). 
\end{align*} 
Therefore, 
\begin{align}\label{logsectionnorm}
\log(||\eta||^2_{\hat{h}^*})= -2\log(u)+ \log(\dis\sum_{i=1}^{n-1}\sum_{j=1}^{n-1}\bi( (u\delta_{ij}+\bar{\zeta_i} \zeta_j)\xi_i\bar{\xi}_j
+
\frac{t_{\infty}}{4\pi}\xi_j \bar{\xi}_n \bar{\zeta_{j}}
+
\frac{t_{\infty}}{4\pi} \bar{\xi}_j \xi_n \zeta_{j}
\bi) 
+
\frac{t^2_{\infty}}{16\pi}|\xi_n|^2).
\end{align}

To check that the $1$-form $\del \log(||\eta||^2_{\hat{h}^*})$ is locally integrable, we only need to check that the the function $\dfrac{1}{|q|\log|q|}$ is locally integrable around $q=0.$ This follows from $$\dis\int_{|q|<\epsilon}\dfrac{dq\wedge d\bar{q}}{|q|\log|q|}=\int_{0}^{2\pi}\int_{0}^{\epsilon}\dfrac{2 r^{\frac{1}{2}}dr d\theta}{\log(r)}<\infty,$$ 
for small enough $\epsilon>0.$






To check that the $(1,1)$-form $\delb \del\log(||\eta||^{2}_{\hat{h}^*})$ 
is locally integrable around $q=0,$ we should only check that the function $\dfrac{1}{|q|^2\log^2|q|^2}$ is locally integrable around $q=0.$ This follows from 
$$\dis\int_{|q|<\epsilon}\dfrac{dq\wedge d\bar{q}}{|q|^2\log^2|q|}=\int_{0}^{2\pi}\int_{0}^{\epsilon}\dfrac{ dr d\theta}{r\log^2(r)}<\infty,$$ 
for small enough $\epsilon>0.$ When $\zeta_1,...,\zeta_n$ and $\xi_1,...,\xi_n$ are bounded and $q\to 0,$ using \eqref{logsectionnorm} we have the the asymptotic relation  
$$
\dfrac{\log(||\eta||^{2}_{\hat{h}^*})}{\log|q|}
\sim\dfrac{c\log(u)}{\log|q|}
\sim\dfrac{c\log(-\log|q|)}{\log|q|}
\sim 0,
$$
 where $c\in \R.$ Hence, we can apply Proposition \ref{Current} to $\log(||\eta||^2_{\hat{h}^*})$. 


\end{proof}
\section{$\Q$-vector bundles and Base loci}\label{Sec4}

The goal of this section is to prove formal properties of $\Q$-vector bundles which will be used in \textsection \ref{sec3} and \textsection \ref{sec5} to study the twists of $\Cot$ and $\lgcot.$ To this end, we use the definitions and properties of the base loci of a vector bundle which has been systematically studied in \cite{bauer2015positivity}. We also use the definitions and properties of $\Q$-vector bundles from \cite[section 6]{Lazarsfeld2}.     

Let $E$ be a vector bundle over a projective variety $\Yb,$ and $D$ be an integral divisor on $\Yb.$ The base locus of $E$ is defined to be the subset
$$
\textbf{Bs}(E)=\{y\in \Yb|H^0(\Yb,E)\rightarrow E(y)\mbox{ is not surjective}\},
$$
and the stable base loci of $E$ is defined to be the algebraic subset 

$$\dis \mathbb{B}(E)=\bigcap_{m\ge 0} \textbf{Bs}(S^{m}E).$$
Let $r$ be a rational number. 
\begin{definition} A $\Q$-vector bundle $E \la rD \ra$ on $\Yb$ is a pair consisting of a vector bundle $E$ on $\Yb$ and a $\Q$-divisor $rD\in \mathrm{Div}_{\Q}(\Yb).$ A $\Q$-isomorphism of $\Q$-vector bundles is the equivalence relation generated by considering $E \la D'+ rD \ra$ to be equivalent to 
$E \otimes O_{\Yb}(D') \la rD \ra,$
where $D'$ is an integral divisor.     

\end{definition}
The  formal symbol $E \la rD \ra$ is intended to say that we are twisting the vector bundle $E$ by a $\Q$-divisor $rD.$ The symmetric power $S ^ m \bi(E \la r D \ra \bi)$ denotes the  $\Q$-vector bundle $S^m (E) \la m r D \ra.$ It is easy to observe that every $\Q$-vector bundle has a symmetric power which is $\Q$-isomorphic to a vector bundle.
Let $a$ and $b$ be integers such that $b$ is positive. We will define the base loci of the $\Q$-vector bundle $E \la  \frac{a}{b} D \ra$ to be the following set 

$$\dis \mathbb{B}(E \la \frac{a}{b} D \ra)=\bigcap_{m>0} \textbf{Bs}\bi(S^{ m b}E\otimes O_{\Yb}( m a D) \bi ),$$
which does not depend on the choice of $a,b.$ Let $A$ be an ample divisor on $\Yb$. The augmented base locus of $E$ is the algebraic subset

$$\mathbb{B}_{+}( E )=\bigcap_{r \in \Q_{+} }\mathbb{B} (E \la - r A \ra ),$$
which does not depend on the choice of the ample divisor $A.$

\begin{definition} \label{PosTwist}
We say that $E \la r D\ra$ is ample (nef) if one of the following equivalent properties holds:
\begin{enumerate} [label=(\roman*)]
    \item The $\Q$-divisor $O_{\Proj(E)}(1)+\pi^*(r D)$ is ample (nef) on $\Proj(E).$
    \item $E \la r D\ra$ has a symmetric power which is $\Q$-isomorphic to an ample (a nef) vector bundle.
    \item If $r=a/b$ and $b$ is positive, then the vector bundle $S^b E \otimes O_{\Yb}( a D )$ is ample (nef). 
\end{enumerate}
\end{definition}
The equivalency of these definitions is checked in \cite[Lemma 6.2.8]{lazarsfeld1}. 
\begin{Lemma}\label{ampletwist} If $E$ is a nef vector bundle  and $rD$ is an ample $\Q$-divisor, then $E \la  r D \ra$ is an ample $\Q$-vector bundle.   
\end{Lemma}
\begin{proof} Take $a,b \in \Z$ such that $r=a/b$ and $b$ is positive. As $rD$ is ample, the integral divisor $aD$ is ample.    
Since $E$ is nef, $S^{b} E$ is nef. Combining these two we get that $S^b E \otimes O_{\Yb}(a D )$ is an ample vector bundle. Hence, $E \la rD \ra$ is ample.  
\end{proof}
\begin{Lemma} \label{NefSequence}
Suppose $0\to G \to E \to F \to 0$ is an exact sequence of vector bundles on a projective variety $\Yb$. The following hold: 
\begin{enumerate} [label=(\roman*)] 
    \item If $E \la r D \ra$ is ample (nef), then $F \la r D \ra $ is ample (nef).
    \item If $G \la r D \ra$ and $F \la r D \ra$ are ample, then $E \la rD \ra$ is ample.
\end{enumerate}

\end{Lemma}
\begin{proof}
(i) Since $\Proj (E)$ parametrizes
one-dimensional quotients, the surjection $E \to F$ corresponds to an inclusion $\Proj(E)\subset \Proj(F)$, such that the restriction of $O_{\Proj(E)}(1)$ is $O_{\Proj(F)}(1)$. Hence, the ampleness (nefness) of $O_{\Proj(E)}(1)+\pi^*(rD)$ implies the ampleness (nefness) of $O_{\Proj(F)}(1)+\pi^*(rD)$. 

(ii) See \cite[Lemma 6.2.8]{Lazarsfeld2}.
\end{proof}
Note that the above mentioned definitions for a $\Q$-vector bundle $E \la r D \ra$ agree with the usual definitions for a vector bundle when a vector bundle $E$ is considered as the $\Q$-vector bundle $E \la 0 \ra.$

\begin{definition} \label{ModDDef}
A vector bundle $E$ over $\Yb$ is said to be ample modulo $D$ if for every coherent sheaf $\mathscr{F}$ over $\Yb$, there exists an $m_{0}>0$ such that if $m>m_{0}$, then for every $y\in \Yb\setminus{D}$, the fiber $ \mathscr{F} \otimes S^{m}(E)_{|y}$ is generated by $H^{0}\bi(Y, \mathscr{F} \otimes S^{m}(E)\bi)$.
\end{definition}

 There is an obvious relation between the notion of ampleness modulo a divisor and the augmented base loci:  

\begin{Proposition} \label{Aug} If the vector bundle $E$ is ample modulo $D$, then the augmented base loci $\mathbb{B}_{+}(E)$ contained in $D$ and $\mathbb{B}_{+}\bi(O_{\Proj(E)}(1)\bi)$ is contained in $\pi^{*}(D).$ 
\end{Proposition}
\begin{proof}
Pick an arbitrary point $y\in \Yb\setminus D$. To show $\mathbb{B}_{+}(E)\subset D$, it is enough to show that $y\notin \mathbb{B}_{+}(E).$  Let $A$ be an ample line bundle. By Definition \ref{ModDDef}, there exist $n$ such that $(S^n E -A)_{|y}$ is generated by $H^0(Y, S^n E -A)$ and therefore $y\notin \mathbb{B}(E\la -\frac{1}{n}A \ra ).$  Consequently, $y\notin \mathbb{B}_{+}(E)$.

 \cite[Proposition 3.2]{bauer2015positivity} tells us that  $\mathbb{B}_{+} \bi (O_{\Proj(E)}(1) \bi ) \subset \pi^{-1}(\mathbb{B}_{+}(E))$ and therefore the previous part gives that $\mathbb{B}_{+}(O_{\Proj(E)}(1))\subset \pi^{*}(D).$ 

\end{proof}

  Suppose $L$ is a line bundle on a projective variety $\Yb$. It is well-known that $L$ is big if and only if $\mathbb{B}_{+}(L)\neq X$. This follows that if $E$ is ample modulo $D$, then $E$ is big in the sense that $O_{\Proj(E)}(1)$ is big. Another well-know fact is that $L$ is ample if and only if $\mathbb{B}_{+}(L)= \emptyset .$ For vector bundles, we will use the following lemma to go from the ampleness modulo $D$ to the ampleness:

 \begin{Lemma}\label{ModDonD}
    If $E$ is ample modulo $D$ and $E_{|D}$ is ample, then $E$ is ample.
 \end{Lemma}
\begin{proof}
Let $L$ be the line bundle $O_{\Proj(E)}(1)$ on $\Proj(E).$ To conclude ampleness it is sufficient to prove that for any subvariety $V\subset \Proj(E)$ with dimension $m,$ one has $c_1(L)^{\operatorname{dim} V} \cdot V > 0.$ We prove it by induction on $m:$ 

First, we show that for every irreducible curve $C$ on $\Proj (E),$ we have $L\cdot C>0.$ If $C \subset \pi^*(D),$ then $C \cdot L > 0$ because $L_{|\pi^*(D)}$ is ample. If $C$ is not contained in $D,$ then there is a point $x\in C$ such that $x \not\in \pi^*(D).$ Since $E$ is ample modulo $D$, Proposition \ref{Aug} gives that $\mathbb{B}_{+}( L ) \subset \pi^*(D)$ and therefore $x \not \in \mathbb{B}_{+}( L ).$ It means that there exist $a,b \in \Z^{+}$ such that $x \not \in \textbf{Bs}(b L - a A ),$ where $A$ is an ample divisor on $\Proj(E).$ Consequently, the zero locus of $b L - a A$ is not contained in $C$ and it follows that $C$ intersects transversely with $b L - a A.$ Hence.  $C \cdot (b L - a A)  \ge 0$ which gives that $C \cdot L >0.$  

Similarly, when $m>1,$ there are two cases to consider. If $V \subset \pi^{*}(D)$, then the result comes from the fact that $E_{|D}$ is ample. If  $V \not\subset \pi^{*}(D)$, then $V \not \subset \mathbb{B}_{+}( L ).$ Therefore, there exists $a,b \in \Z^{+}$ and a section $s \in H^0(bL - aA)$ such that $V \not\subset \{s = 0\}$. Let $W$ be the $(m- 1)-$cycle on $V$ associated with the schematic locus $\{s_{|V}=0\}$. It is either $0$ or a formal sum of
$(m- 1)-$dimensional subvarieties with positive coefficients. Therefore, one has
$$(bL - aA)\cdot V \sim _{rat} W,$$
and therefore $(bL)\cdot  V \sim_{rat}  W+ (aA) \cdot V.$  As A is ample, the class $A\cdot V$ is provided by an
$(m-1)-$cycle with only positive coefficients. Since
$$
b L^{m}\cdot V= L^{m-1}\cdot(W+aA\cdot V)
$$
and the second factor on the right hand side is a $(m-1)-$dimensional cycle with only
positive coefficients, so we get the result by induction.

\end{proof}

\begin{definition}
We say that $E\la r D \ra$ is ample modulo $D$ if one of the following two equivalent properties holds:
\begin{enumerate} [label=(\roman*)]
\item  $E\la r D \ra$ has a symmetric power which is $\Q$-isomorphic to an ample modulo $D$ vector bundle.
\item  If $r=a/b$ and $b$ is positive, then the vector bundle $S^b E \otimes O_{\Yb}( a D)$ is ample modulo $D$. 
\end{enumerate}
 
\end{definition}

\begin{definition} Consider two $\Q-$ vector bundles $E \la \frac{a}{b} D\ra $ and $F \la \frac{a'}{b} D\ra, $ where $r$ and $s$ are two rational numbers. We say that $E \la r D\ra$ is a $\Q$-subsheaf of $F \la s D\ra $ if for every integer $a,a'$ and every positive integer $b$ such that $r=\frac{a}{b}$ and $s=\frac{a'}{b},$ the vector bundle $S^b E \otimes O_{\Yb}( a D )$ is a subsheaf of $S^ {b} F \otimes O_{\Yb}( a' D ).$ 
\end{definition}

\begin{Lemma} \label{PositiveSub}
Let $E$ and $F$ be vector bundles on $\Yb$ and they are isomorphic over $Y:=\Yb \setminus{D}.$
Suppose $E \la r D \ra$ is a $\Q$-subsheaf of $F \la s D \ra ,$ and $E \la r D \ra$
is ample modulo $D$. Then $F \la s D \ra$ is ample modulo $D.$
\end{Lemma}

\begin{proof} Take $a,a',b \in \Z$ such that $a/b=r,a'/b'=s$ and $b$ is positive. Since the $\Q$-vector bundle $E \la \frac{a}{b} D \ra$ is ample modulo $D,$ there exists an integer $l\in \Z$ such that the vector bundle $S^{b l} E \otimes O_{\Yb}( a l D )$ is ample modulo $D.$    
It means that for a given coherent sheaves $\mathscr{F}$ over $\Yb$ there exists $m_{0}>0$ such that for every $m>m_{0}$ and every $y\in Y$,  $\bi(\mathscr{F} \otimes S^{b l m} E \otimes O_{\Yb}(a l D) \bi )_{|y}$ is generated by $H^{0}( \Yb, \mathscr{F} \otimes S^{b l m} E \otimes O_{\Yb}(a l D) )_{|y}$. Because $E \la r D \ra$ is a $\Q$-subsheaf of $F \la s D \ra$, we obtain that $\mathscr{F} \otimes S^{b l m} E \otimes O_{\Yb}(a l D))$ is a subsheaf of $\mathscr{F} \otimes S^{b l m} F \otimes O_{\Yb}(a' l D))$. This implies that 
 $$H^{0}\bi(\Yb,\mathscr{F} \otimes S^{b l m} E \otimes O_{\Yb}(a l D) \bi)
 \subseteq
 H^{0}\bi(\Yb,\mathscr{F} \otimes S^{b l m} F \otimes O_{\Yb}(a' l D) \bi).$$
 
 Since these two sheaves restricts to the same sheaf on $Y$, for every $y\in Y$ and for every $m>m_{0}$, $ \mathscr{F} \otimes S^{b l m} E \otimes O_{\Yb}(a' l D)$ is generated by 
$H^{0}\bi(\Xb,F\otimes S^m E_{2} \bi).$
\end{proof}

\begin{Lemma} \label{smallertwist}
Let $r \in \Q^+$ be a positive rational number. If $E \la - r D\ra$ is ample modulo $D$, then $E \la - s D \ra$ is ample modulo $D$ for any rational $s<r.$
\end{Lemma}
\begin{proof}
Let $a,a', b \in \Z$ be positive integers such $a/b= r$ and $a'/b= s.$ Since $s<r,$ we get that $a' < a$ and therefore $S^b E \otimes_{\Yb}(-a D )$ is a subsheaf of $S^b E \otimes_{\Yb}(-a' D )$ and the claim follows from Lemma \ref{PositiveSub}. 
\end{proof}

\begin{Lemma} \label{neat}
Suppose $f:Y' \longrightarrow Y$ is a finite surjective map and $E$ is a vector bundle on $Y.$ Then, $f^* E$ is semi-ample, then $E$ is so.   
\end{Lemma}
\begin{proof} By passing to the projective bundles, the lemma follows from the corresponding facts for line bundles ( see \cite[1.20]{Fujita}) 
\end{proof}

\section{Positivity of $\lgcot$}
\label{sec3}
Throughout this section, we adopt the convention that $\Gamma$ is a torsion-free lattice in $\operatorname{PU}(n,1),$ $X=\Gamma \backslash\mathbb{B}^n $ is a complex hyperbolic manifold with cusps, $\Xb$ is the toroidal compactification of $X$ and $d$ is the uniform depth of cusps. Additionally, we assume that $\Xb$ does not have any orbifold point. 

The goal of this section is to prove the ampleness of the twisted logarithmic cotangent bundle as stated in Theorem \ref{SlopLog}. 
To prove this theorem, we need to construct a singular hermitian metric using the properties of the Bergman metric proved in the previous section, namely Proposition \ref{Positive form} and Proposition \ref{Good metric}.  
By virtue of Theorem \ref{SlopLog}, we will conclude the positivity properties of $\lgcot$ stated in Corollary \ref{logPositivity}.

Suppose $\Yb$ is a smooth projective variety and $D$ is divisor on $\Yb$. As proved in \cite{Ascher, Brotbek}, if $dim \ \Yb>1$, the logarithmic cotangent bundle  $\Omega^1_{\Yb}(\log(D))$ is never ample because its restriction to $D$ is an extension of the trivial bundle. Therefore, to describe a positivity properties of $\Omega^1_{\Yb}(\log(D))$ we need a weaker positivity notion.

In the case of the toroidal compactification $\Xb$ of $X$, the first step toward understanding the positivity of the cotangent bundles is  to study the positively of the cotangent bundle restricted to the boundary divisor $D$. The connected components of $D$ are \'etale quotients of abelian varieties whose conormal bundles are ample and therefore considering the conormal bundle exact sequence
\begin{align} \label{Conormal}
0\longrightarrow O_{D}(-D)\longrightarrow \Omega^{1} _{\Xb|D} \longrightarrow \Omega^{1}_{D} \longrightarrow 0,   
\end{align}
on $D$, one can observe that $\Omega^1_{\Xb|D}$ is an extension of a vector bundle with vanishing Chern classes, by an ample line bundle. Moreover, we will prove that $ \Omega^{1} _{\Xb|D}$ is semi-ample in the sense that $O(1)$ on $\Proj( \Omega^{1} _{\Xb|D})$ is semi-ample. To this end, we need the following lemma:
\begin{Lemma} \label{Semiampleness}
Suppose $0\to G \to E \to F \to 0$ is an exact sequence of vector bundles on a projective variety $\Yb$ such that $H^1(\Yb,G )=0.$ If $S^m G$ and $F$ are globally generated, then $O_{\Proj(E)}(m)$ is globally generated. 
\end{Lemma}
\begin{proof}
 
$H^1(\Yb,G )=0$ gives the exact sequence of global sections:
\begin{align*}
    0\longrightarrow H^{0}(\Yb,G)\longrightarrow H^{0}(\Yb, E)\longrightarrow H^{0}(\Yb,F) \longrightarrow 0.
\end{align*}
This implies that $H^0(\Yb,S^m G)$ injects into $H^0(\Yb,S^m F)$ and $H^{0}(\Yb, E)$ surjects onto $H^{0}(\Yb,F).$

To prove that $O_{\Proj(E)}(m)$ is globally generated, we need to show that for every $p \in \Proj(E),$ the fiber $O_{\Proj(E)}(m)_{|p}$ is generated by global sections of $O_{\Proj(E)}(m).$ 
Suppose $p \in \Proj(E)$ is determined by $y \in \Yb$ and an one-dimensional quotient $E_p \to L_p.$
Consider the map $\eta_{y}: G_y \to L_y.$ We may have two cases: 
\begin{enumerate}
    \item $\eta_{y} : G_y \to L_y$ is the zero map. In this case, $\eta_{y}$ factors through $F_y.$ On the other hand, $H^{0}(\Yb, E)$ surjects onto $H^{0}(\Yb,F),$ and $F$ is globally generated. Since 
    $$H^{0}(\Yb, E) \cong H^{0}(\Proj(E), O_{\Proj(E)}(1) ), $$ the fiber $O_{\Proj(E)}(1)_{|p}$ is generated by the global sections and therefore $O_{\Proj(E)}(m)_{|p}$ is generated by global sections.  
    \item $\eta_{y} : G_y \to L_y $ is a non-zero map. Since $H^0(\Yb,S^m G)$ injects into $H^0(\Yb,S^m E)$ and $S^m G$ is globally generated, the global sections $ H^{0}(\Proj(E), O_{\Proj(E)}(m)) \cong H^0(\Yb,S^m E)$ generates the fiber $O_{\Proj(E)}(m)_{|p}.$

    Hence, $O_{\Proj(E)}(m)$ is globally generated, as desired. 
\end{enumerate}

\end{proof}

Now, we can study the positivity of the twisted cotangent bundle restrict to the boundary divisor:
\begin{Proposition} 
\label{NefBoundry}
Suppose that the dimension of $X$ is greater than $1$ and $r$ is a rational number. Consider the $\Q$-vector bundles  
$$E_{r}:=\Omega^{1} _{\Xb|D } \la - r D_{|D} \ra .$$
The following hold:
\begin{enumerate} [label=(\roman*)]
     \item If $r=0$, then $E_r=\Omega^{1} _{\Xb|D}$ is semi-ample, but not ample. \label{NefBoundryB}
     
    \item If $r>0$, then $E_r$ is ample. 
   
\end{enumerate}
\end{Proposition}

\begin{proof} 
 (i) Let $D_i$ be a connected component of $D$. We can write the exact sequence \ref{Conormal} on $D_i:$ 
\begin{align*} 
0\longrightarrow O_{D_i}(-D_i)\longrightarrow E_{0|D_i} \longrightarrow \Omega^{1}_{D_i} \longrightarrow 0.
\end{align*}

As $D_i$ is an \'etale quotient of abelian  variety, there is a finite \'etale map $f: D' \longrightarrow D_i,$ where $D'$ is an abelian variety. We pull back the previous exact sequence to $D':$

\begin{align*} 
0 \longrightarrow
f^* O_{D_i}(-D_i) \overset{\phi}{\longrightarrow}
f^* E_{0|D_i} \longrightarrow
f^* \Omega^{1}_{D_i}
\longrightarrow 0.
\end{align*}

 As $O_{D_i}(-D_i)$ is ample and $f$ is finite, $f^* O_{D_i}(-D_i)$ is ample. As $f$ is an \'etale map $f^* \Omega^1_{D_i} \cong \Omega^1_{D'}$ which is trivial because $D'$ is an abelian variety. Since $f^*E_{0|D_{i}}$ has a trivial quotient, $E_{0|D_{i}}$ can not be ample.
 
 To show semi-ampleness, note that $\Omega^1_{D_i}$ is in particular globally generated. Also, $f^* O_{D_i}(-D_i)$ is ample because $-D_{i|D_i}$ is ample and $f$ is finite. Therefore, there exist $m \in \Z^{+}$ such that $m f^* O_{D_i}(-D_i)$ is globally generated on $f^{-1}(D_i).$ Additionally, the Kodaira's vanishing theorem gives that $H^1(f^{-1}(D_i),f^* O_{D_i}(-D_i))=0.$ Thanks to Lemma \ref{Semiampleness}, we can conclude that $f^{*} E_{0|D_i}$ is semi-ample and therefore Lemma \ref{neat}, implies that $E_{0|D_i}$ is semi-ample.     
 
(ii) Since $E_{0|D_i}$ is in particular nef by (i) and the conormal bundle $-D_{D}$ is ample, we can conclude from Lemma \ref{ampletwist} that $E_{r}$ is ample.   
 
\end{proof} 

Recall the following exact sequence: 
\begin{align}\label{lgcotseq}
    0 \longrightarrow \Omega^{1}_{\Xb} \longrightarrow \Omega^{1}_{\Xb}(\log(D)) \overset{\operatorname{res}}{\longrightarrow} O_{D} \longrightarrow 0,
\end{align}
where the morphism $\operatorname{res}$ is the residue map sending $\dis\sum_{i=1}^{n-1}f_id\zeta_i+g \frac{dq}{q}$ to $g.$

\begin{Proposition}\label{AmpleLog}
Suppose that the dimension of $X$ is greater than $1$ and $r$ is a rational number. Consider the $\Q$-vector bundles  
$$F_{r}:=\lgcot_{|D} \la  - r D_{ | D } \ra .$$
The following hold: 
\begin{enumerate} [label=(\roman*)] 
    \item If $r=0$, then $F_r=\Omega^{1}_{\Xb}(\log(D)) _{|D}$ is nef, but not ample. 
    \item If $r>0$, then $F_r$ is ample.
\end{enumerate} 
\end{Proposition}

\begin{proof}
 Let $D_i$ be a connected component of $D$. Restricting the exact sequence \eqref{lgcotseq} to $D_i$ gives the exact sequence: 
\begin{align*}
    \Omega^{1}_{\Xb|D_i} \overset{\phi_1}{\longrightarrow} F_{0|D_i} \overset{\phi_2}{\longrightarrow} O_{D_i} \longrightarrow 0,
\end{align*}
where $\phi_1$ is the inclusion and $\phi_2$ is the residue map sending $\dis\sum_{i=1}^{n-1}f_idz_i+g \frac{dq}{q}$ to $g_{|D_i}$
on an open set $U$ on which the local coordinate of $D_i$ is $(q=0)$. Therefore, we get the exact sequence 
\begin{align*}
0 \longrightarrow
    \operatorname{Im} (\phi_1) \longrightarrow F_{0|D_i} \overset{\phi_2}{\longrightarrow} O_{D_i} \longrightarrow 0.
\end{align*}
Note that $ \operatorname{Im}(\phi_1)$ is a quotient of $\Omega^{1}_{\Xb|D_i}$ which is in particular nef by Proposition \ref{NefBoundry}. Therefore, $ \operatorname{Im}(\phi_1)$ is nef. 

(i) Since $F_{0|D_i}$ admits a trivial quotient, it is not ample. However, it is an extension of a nef bundle by a nef bundle; therefore it is nef. 

(ii) Since $F_{0|D_i}$ is nef and $- r D_{i|D_{i}}$ is ample for $r>0,$ it follows from Lemma \ref{ampletwist} that $F_r$ is ample when $r>0.$ 
 
\end{proof}

Another ingredient we need in order to construct a singular hermitian metric is an appropriate weight function. Roughly speaking, this weight function will be a plurisubharmonic function, supported on a horoball with the largest possible Lelong number on the boundary. The desired wight function on the hoorball around the cusp $c_i$ will be constructed using the following Lemma: 

\begin{Lemma} \label{Weight}
For a positive real number $u_i$ and sufficiently small $\epsilon>0$, there exists a \\ $C^2$ function $\rho_{i}:\mathbb{R}_{>0}\to \mathbb{R}$ and a constant number $c$ satisfying the following properties:
\begin{enumerate}
    \item $\rho_{i}(u)= -\log(u)+c$ on $(0,u_i+\epsilon']$, where $\epsilon'$ is a  positive number depending on $\epsilon$.
    
    \item $i\del \delb \rho_{i}\ge 0$
    \end{enumerate} 
    
    Let $u=-\frac{t_i}{2\pi}\cdot \log|q|-|\zeta|^2$, where $t_i\in \mathbb{R}^{+}$, and $(\zeta, q)\in \C^{n-1} \times \C$.
    \begin{enumerate}[resume]
    \item When $(\zeta,q)$ varies on a compact set with non-empty intersection with $q=0$, 
    $$\dis \lim_{q\to 0 }\frac{ \rho_{i}(u)}{\log|q|}=\frac{1}{2\pi}(d_i-\epsilon),$$
    where $d_i=\frac{t_i}{u_i}.$  
    
    \item The forms $\delb\bi(\rho(u)- \frac{1}{2\pi}(d_i-\epsilon)\log|q|\bi)$ and $\del\delb\bi(\rho(u)\bi)$ are locally integrable on the set $\{(\zeta,q)\in \C^{n-1}\times \C\ | \ u> u_0+\epsilon'\}$.
\end{enumerate}
\end{Lemma}
\begin{proof} 
Let $\epsilon>0$ be small enough such that there exists $\epsilon'\in(0,\frac{t_i}{4\pi})$ satisfying
$\frac{t_i}{u_i}-\frac{t_i}{u_i+2\epsilon'}=\epsilon$  and define $\rho_{i}:\mathbb{R}_{>0}\to \mathbb{R}$ be the function given by
\begin{equation*}
\rho_{i}(u) = \left\{
        \begin{array}{ll}
            -\log(u)+c_0 & \quad 0<u< u_i+\epsilon' ,\\
            \\
            \dfrac{-u}{u_i+2\epsilon'}+c_1e^{-a(u-u_i-\epsilon')} & \quad u \ge u_i+\epsilon' ,\\
        
        \end{array}
    \right.
\end{equation*}
where 
$a=\frac{1}{\epsilon'}+\frac{1}{u_i+\epsilon'},
c_1=\bi(\frac{\epsilon'
}{u_i+2\epsilon'})^2,$
and 
$c_0= -\log(u_i+\epsilon')-c_1-\dfrac{u_i+\epsilon'}{u_i+2\epsilon'}$. Since $a\cdot c_1=\frac{\epsilon'}{(u_i+\epsilon')(u_i+2\epsilon')},$ and 
$a^2c_1=\frac{1}{(u_i+\epsilon')^2}$
its straightforward to see that the function $\rho_i$ is $C^2$.

As $-u$ is a plurisubharmonic function, and $c_1$ and $a$ are positive,  $-\log(u)+c_0$, $\frac{-u}{u_i+2\epsilon'}+c_1e^{-(u-u_i-\epsilon')}$ are plurisubharmonic because the function $-\log(-x)$ when $x<0$ and the function $e^{bx}$  for every $b>0$ are monotonically increasing and convex. Hence, $\rho_i(u)$ satisfies the second properties.

$\rho_i(u)$ satisfies the third condition since
\begin{align*}
\lim_{ q\to 0}\dfrac{ \rho_{i}(u)}{\log|q|
}= \frac{1}{2\pi} \frac{t_i}{u_i+2\epsilon'}
=\frac{1}{2\pi} (\frac{t_i}{u_i}-\epsilon)
=\frac{1}{2\pi} (d_i-\epsilon).
\end{align*}

To check the forth condition note that 
\begin{align*}
e^{-a(u-u_i-\epsilon')}=c'\cdot e^{-a|\zeta|^2}\cdot |q|^{l},    
\end{align*}
where $l=\dfrac{a\cdot t_i}{2\pi}$ and $c'$ is a constant. The inequality $a> \frac{1}{\epsilon'}> \frac{4\pi}{t_i}$ gives that $l>2$, the function $|q|^l$, and the forms $\delb(|q|^{l}), \del(|q|^{l})$ and $\del\delb(|q|^{l})$ are locally integrable and therefore letting $c''= \frac{1}{u_0+2\epsilon'}$,

$$\del\delb\bi(\rho_i(u)\bi)=\del\delb(c'' \cdot|\zeta|^2)+\del\delb(c'e^{-a|\zeta|^2}|q|^l)$$
is a locally integrable form. 
On the other hand, $\frac{-u}{u_i+2\epsilon'}= (d_i-\epsilon)\log|q|+\frac{|\zeta|^2}{u_i+2\epsilon'},$ in other words, $$\delb\B(\frac{-u}{u_i+2\epsilon'}-(d-\epsilon)\log|q|\B)= \delb(c''\cdot |\zeta|^2),$$ which is a locally integrable form. Hence, 
$\delb\bi(\rho(u)- (d_i-\epsilon)\log|q|\bi)$ is locally integrable on $\{(\zeta,q)\in \C^{n-1}\times \C\ | \ u> u_0+\epsilon'\}$ as desired. 
\end{proof}

The main difficulty to prove Theorem \ref{SlopLog} is to prove the following proposition. In the proof, we construct a singular hermitian metric by Preposition \ref{Current} and Lemma \ref{Weight}:  

\begin{Proposition} \label{Positve intersection}
For a sufficiently small $\epsilon>0,$ with $r= (d-\epsilon)/4\pi$ being rational, the $\Q$-vector bundle
$$\Cot\bi(\log(D)\bi) \la -r D \ra$$
is ample.
\end{Proposition}

\begin{proof}
Denote $\Cot\bi(\log(D) \bi)$ by $F$.  
Suppose $\tildeb{D}$ is the pullback of $D$ by the natural projection 
$$\pi: \Proj(F)\longrightarrow \Xb.$$ 

The goal is to show that the $\Q$-line bundle $O_{\Proj(F)}(1)-r \tildeb{D}$ is ample, i.e., it intersects positively with every subvariety of $\Proj(F)$ . Since $-r=-(d-\epsilon)<0$ for a sufficiently small $\epsilon>0$, by Proposition \ref{AmpleLog}, it is enough to show that $O_{\Proj(F)}(1)-r \tildeb{D}$ intersects positively with every subvariety of $\Proj(F)$ that is not contained in $\tildeb{D}$ but possibly intersects with $\tildeb{D}$. Let $V$ be such a subvariety.

Let $a$ and $b$ be positive integer such that $a/b= r$. We show that the line bundle $L=O_{\Proj(F)}(b) \otimes O_{\Proj(F)}(- a \tildeb{D})$ intersects positively with $V$ by constructing a singular hermitian metric $\bar{h}$ on $L$ such that $c_1(L, \bar{h})^{dim(V)} \cdot V > 0.$ 

Suppose $D_i$ is the component of $D$ compactifing the cusp $c_i$ and $D_i$ is given by $q_{i}=0$ on the horoball $H_{i}(u_i),$ where $u_i=t_i/d.$ Let $\tildeb{D}_i=\pi^{*}(D_i)$ and $\tildeb{H}_{i}(u_i)=\pi^{-1}\bi(H_{i}(u_i)\bi).$ By Shimuzu's lemma the horoballs are disjoint and therefore $L_{|\tildeb{H}_i(u_i)}$ naturally isomorphic to $O_{\Proj(F)}(b)\otimes O_{\Proj(F)}(- a \tildeb{D}_{i}).$

 Substituting $u_i$ and $\epsilon$ to Lemma \ref{Weight}, we obtain a function $\rho_i(u)$ and a constant $c$ satisfying the properties in Lemma \ref{Weight}. Now, we define a singular hermitian metric 
$$\bar{h}(s^b \otimes \beta):= \exp\B(b \cdot \bi (-\rho_{i}(u)-\log(u)-c\bi )\B ) \hat{h}(s)^b$$
on $O_{\Proj(F)}(b)\otimes O_{\Proj(F)}(- a \tildeb{D}_{i})$, where $s^b$ is a section of $O_{\Proj(F)}(b)$ and $\beta$ is the canonical rational section of $-a \tildeb{D}_{i}$ corresponding to $1$. As horoballs are disjoint, Lemma \ref{Weight} implies that $\bar{h}$ is a well-defined metric on $L$ and on the complements of horoballs it equals to $\hat{h}^{b}$. Evaluating  $\bar{h}$ at a local generator of $L$ on $\tildeb{
H}_i(u_i)$ gives

$$\bar{h}(s^b \otimes \beta q_i^{a})=\exp\Big( b \cdot \bi(-\rho_{i}(u)+2r \log|q_i|-\log(u)-c \bi) \Big)\hat{h}(s).$$

Now, we can compute the curvature current on the horoball $\tildeb{H}_i(u_i)$:

\begin{align}
\bi[c_{1}(L, \bar{h}) \bi]
=\frac{i}{2\pi}\Big( 
&\bi[\del\delb\bi(\rho_i(u)-2r \log|q_i|\bi)\bi] \nonumber \\
&+
[\del\delb\log(u)]
-
[\del\delb\log(\hat{h}(s))]\Big) \nonumber
\end{align}
By Proposition \ref{Good metric}, the current
$[\del\delb \log(\hat{h}(s))]$ is represented by the form $\del\delb\log(\hat{h}(s))$. Applying Proposition $\ref{Current}$ to $\Phi_1(\zeta,|q_i|)=\exp\bi(\rho_i(u)-2 r \log|q_i|\bi)$ and 
$\Phi_2(\zeta,|q_i|)=\exp\bi(-\log(u)\bi)$ gives that 
$$\bi[\del\delb\bi(\rho_i(u)-2r \log|q_i|)\bi)\bi] +[\del\delb\log(u)]$$
is represented by the form $\del\delb\bi(\rho_i(u)-2r  \log|q_i|\bi) +\del\delb\log(u)$. 
By Proposition \ref{Positive form}, we observe that $\del\delb\log(u)
-
\del\delb\log(\hat{h}(s))$ is a semi-positive form. On the other hand, Lemma \ref{Weight} 
gives that the form 
$\del\delb\bi(\rho_i(u)-2r\log|q_i|\bi)= i\del\delb \rho_i(u)$
is semi-positive. Putting these together, we obtain that $[c_{1}(L, \bar{h})]$ is represented by a semi-positive form on $\tildeb{H}_i(u_i)$. 

Note that by Lemma \ref{Weight} we killed the Lelong number of the current $\bi[\del\delb\bi(\rho_i(u)-2 r\log|q_i|\bi)\bi]$ on $q_i=0,$ and the currents $[\del\delb\log(u)]$ and $[\del\delb \log(\hat{h}(s))]$ also have $0$ Lelong number on $q_i=0$ (see the last part of the proof of Proposition \ref{Good metric}). Therefore, we can apply \cite[Corollary 7.6.]{demailly1992regularization} and obtain that 
$$
\dis\int_{V}\bi[c_{1}(L, \bar{h}) \bi]^{\dim(V)} \ge \int_{V_{|\pi^{-1}(X)}}\bi(\del\delb (\rho_i(u)+\log(u)-\log(\hat{h}(s)) \bi)^{\dim(V)}\ge 0 . 
$$

As $V\cap X\neq \varnothing $, by Proposition \ref{Positive form}, $V$ has a positive intersection with $ L_{|\pi^{-1}(X)}.$ Since $\bar{h}$ does not depend on $V,$ we can conclude that every subvariety $V\subset \Proj(E)$ that is not entirely contained in $\tildeb{D}$ intersects positively with $L$.
\end{proof}

Using Proposition \ref{AmpleLog} together with Proposition \ref{Positve intersection}, we can prove Theorem \ref{SlopLog}:
\begin{theorem}(Theorem \ref{SlopLog})  \label{PositveLog} For every rational $r \in (0,d/4\pi )$, the $\Q$-vector bundle 
$$\lgcot \la - r  D \ra $$
is ample.
\end{theorem}
\begin{proof} Fix $r=a/b \in (0,d/4\pi)$, and choose $\epsilon>0$ such that $r <(d-\epsilon)/4\pi$. Putting Proposition \ref{Positve intersection} and Lemma \ref{smallertwist} together, we get that  
$\lgcot \la -r  D \ra $
is ample module $D,$ i.e. the vector bundle $S^b \lgcot \otimes O_{\Xb}( -a D )$ is ample modulo $D.$ Moreover, it follows from Proposition \ref{AmpleLog} that the restricted $\Q$-vector bundle
$\lgcot_{|D} \la -r  D _{|D} \ra $
is ample and therefore by Lemma \ref{ModDonD} we can conclude that the vector bundle $S^b \lgcot \otimes O_{\Xb}( -a D )$ is ample. 

\end{proof}
Theorem \ref{PositveLog} implies that $\lgcot$ is a limit of ample $\Q$-vector bundles.
Moreover, we show in Corollary \ref{neflog} that $\lgcot$ is ample modulo $D$ and nef. It is previously proved by Cador\'el that $\lgcot$ is big and nef (\cite[Theorem 3]{Cad}).

\begin{Corollary}(Corollary \ref{logPositivity}) \label{neflog}
The logarithmic cotangent bundle $\lgcot$ is ample modulo $D$ and nef. 
\end{Corollary}
\begin{proof}
Denote $\lgcot$ by $F$. Putting Proposition \ref{Positve intersection} and Lemma \ref{smallertwist} together, we get that  
$\lgcot$ is ample module $D.$

It is easy to conclude from Theorem \ref{PositveLog} that $F$ is nef. If $F$ were not nef, then there would be a curve $C\subset \Proj(F)$ such that the intersection
$I:=C\cdot O_{\Proj(F)}(1)$
is negative. Let $\tildeb{D}=\pi
^{*}D.$ Choose a small enough $r >0$ such that $I- r (C\cdot \tildeb{D})<0$  and therefore 
$$C\cdot \bi(O_{\Proj(F)}(1)-r \tildeb{D}\bi)<0,$$
which contradicts the ampleness of $F \la -r D \ra$.

\end{proof}

\section{Positivity of $\Cot$} \label{sec5}
Throughout this section, we denote a complex hyperbolic manifold with cusps by $X$, and its smooth toroidal compactification by $\Xb$.  Additionally, we assume that $\Xb$ does not have any orbifold point.  

The goal of this section is to prove the results on the positivity of the cotangent bundle, namely Theorem \ref{SemiCot} and Theorem \ref{CotPositivity}. Further, we conclude that if the canonical depth of cusps is sufficiently large, then the symmetric differentials on $\Xb$ is finitely generated $\C$-algebra.   

To pass from positivity on the log-cotangent bundle in the previous section to positivity of the cotangent bundle in this section we consider the following exact sequence of coherent sheaves over $\Xb$: 

\begin{align} \label{exactSeq}
    0 \longrightarrow \lgcot\otimes O_{\Xb}(-D)	\overset{\phi}{\longrightarrow} \Omega_{\Xb}^{1} \longrightarrow i_{*}\Omega^{1}_{D} \to 0,
\end{align} 
where $\phi$ sends $\dis \bi(\sum_{i=1}^{n} f_i d\zeta_i+ g \frac{dq}{q}\bi) \otimes q $ to $\dis \sum_{i=1}^{n} q f_i d\zeta_i+ g dq$ on an open set $U$ on which the local coordinate of $D$ is given by $q=0$.

\begin{theorem} \label{SlopeCot} (Theorem \ref{CotPositivity})
 Suppose that the uniform depth of cusps $d$ is greater than $4\pi$. Then, $$\Omega^{1}_{\Xb}\la - r D \ra$$ is ample for all rational $r \in (0,-1+d/4\pi).$

\end{theorem}
\begin{proof}
Let $r=a/b$ with a positive $b$ be a rational number in $(0, -1+d/4\pi)$. Since $d>4\pi,$ Theorem \ref{PositveLog} gives that the $\Q$-vector bundle $ \lgcot \bi\la (- 1 - r ) D \bi \ra$ is ample. In particular, $ \lgcot \bi\la (- 1 - r ) D \bi \ra$ is ample modulo $D.$ Note that the exact sequence \ref{exactSeq} gives that the vector bundle $\lgcot \otimes O_{\Xb}(-D)$ is a subbundle of $\Omega^{1}_{\Xb}.$ Therefore, the $\Q$-vector bundle $ \lgcot \bi\la (- 1 - r ) D \bi \ra$ is a $\Q$-subsheaf of $\Cot \la -r D \ra.$ It follows from Lemma \ref{PositiveSub} that $\Cot \la -r D \ra$ is ample modulo $D,$ that is, the vector bundle $S^b \Cot \otimes O_{\Xb} (- a D ) $ is ample modulo $D.$ On the other hand, the restricted bundle $\Omega^1_{\Xb|D} \la -r D_{|D} \ra$ is ample thanks to Proposition \ref{NefBoundry}. This means that $\bi ( S^b \Cot \otimes O_{\Xb} (- a D) \bi)_{|D} $ is ample. Hence, it follows from Lemma \ref{ModDonD} that $ \Cot \la - r D \ra $ is ample.

\end{proof}

\begin{theorem} \label{CotPos} (Theorem \ref{SemiCot})
 Suppose that the uniform depth of cusps $d$ is greater than $4\pi$. Then, $\Cot$ is ample modulo $D$ and semi-ample. 

\end{theorem}
\begin{proof}
Set $E=\Cot$. Since the uniform depth of all cusps is greater then $4\pi$, Theorem \ref{PositveLog} implies that $ \lgcot\otimes O_{\Xb}(-D)$ is ample. It therefore follows from \ref{exactSeq} and Lemma \ref{PositiveSub} that $E$ is ample modulo $D$.

Let $Y=\Proj(E),$ 
$\pi:\Proj(E)\to \Xb$ be the natural projection and $\tildeb{D}=\pi^*(D)$.  
Given that $E$ is ample modulo $D$ and $E_{|D}$ is semi-ample by Proposition \ref{NefBoundry},
to show that $E$ is semi-ample, we only need to show that there is a large enough $n$ such that setting $L=O_{Y}(n)$
$$H^{0}(Y,L) \overset{\phi}{\longrightarrow} H^{0}(\tildeb{D},L_{|\tildeb{D}})$$
is surjective. Choose $r\in (0,-1+d/4\pi).$ By Theorem \ref{SlopeCot}, $\Omega^{1}_{\Xb} \la - r D \ra$ is ample, i.e., $O_{Y}(1)- r \tildeb{D}$ is ample. We can choose $n$ large enough so that $H^1(Y, L-\Tilde{n} \tildeb{
D})=0,$ where $\Tilde{n}=nr \in \Z.$ Therefore, denoting the restriction of $L$ to the $\Tilde{n}$th  order thickening of $\tildeb{D}$ by $L_{|\Tilde{n}\tildeb{D}}$, we obtain  
$$H^0(Y,L )\overset{\psi}{\longrightarrow} H^0(Y, L_{|\Tilde{n}\tildeb{D}})$$ 
is surjective. Considering the commutative diagram 
 \[\begin{tikzcd}
        & H^0(\tildeb{D}, L_{|\Tilde{n}\tildeb{D}}) \ar[d,"\eta"]     \\
        H^0(Y, L) \ar[ur,"\psi"] \ar[r,"\phi"]  & H^0(\tildeb{D}, L_{|\tildeb{D}}),  
    \end{tikzcd}\]
to conclude that $\phi$ is surjective, it is enough to show that $\eta$ is surjective. To this end, consider the following exact sequence on $\tildeb{D}$:
$$
0\longrightarrow  L(-m\tildeb{D})_{|\tildeb{D}}
\longrightarrow L_{|(m+1)\tildeb{D}}
\longrightarrow L_{|m\tildeb{D}}
\longrightarrow 0,
$$
where $m$ is a positive integer. We prove that for every $m>0$, 
$$
H^0(\tildeb{D}, L_{|(m+1)\tildeb{D}})
\longrightarrow H^0(\tildeb{D},  L_{|m\tildeb{D}})
$$
is surjective by showing that $H^1(\tildeb{D}, L(-m\tildeb{D})_{|\tildeb D})=0.$ This implies that $\eta$ is surjective and therefore $\phi$ is surjective which finishes the proof. 

Let $\varphi:D\longrightarrow Spec(\C).$ Applying the exact sequence of low degrees to the composition of the push-forward functor $$\pi_{*}:Sh(\tildeb{D})\longrightarrow Sh(D)$$ and the global section functor $\varphi_*$ yields 
$$
0\longrightarrow
H^1\bi(D, \pi_{*}(L(-m\tildeb{D})_{|\tildeb D})\bi)\longrightarrow
R^1(\varphi\circ \pi)_{*}(L(-m\tildeb{D})_{|\tildeb D}) \longrightarrow
H^{0}\bi(D, R^{1}\pi_{*}(L(-m\tildeb{D})_{|\tildeb{D}})\bi).
$$
Note that $(\varphi\circ \pi)_{*}$ is the global section functor $H^0(\tildeb{D},-)$ and therefore  
$$R^1(\varphi\circ \pi)_{*}(L(-m\tildeb{D})_{|\tildeb D}) 
= H^1(\tildeb{D},L(-m\tildeb{D})_{|\tildeb D}).$$
Hence, it is sufficient to prove that 
$$H^1\bi(D, \pi_{*}(L(-m\tildeb{D})_{|\tildeb D})\bi)= H^{0}\bi(D, R^{1}\pi_{*}(L(-m\tildeb{D})_{|\tildeb{D}})\bi)=0.$$

To prove $H^{0}\bi(D, R^{1}\pi_{*}(L(-m\tildeb{D})_{|\tildeb{D}})\bi)=0,$ note that $\pi$ is a flat morphism because the fibers of $\pi$ are projective spaces of the same dimension. Since for every $x\in D$, the dimension $h^1(x, \pi_{*}(L-m\tildeb{D})_{x})= h^1\bi(\pi^{-1}(x),O_{\Proj^{n-1}}(n)\bi)$ is constant, Grauert's theorem (\cite[Corollary 12.9]{hartshorne2013algebraic}) gives the isomorphism 
$$R^{1}\pi_{*}(L(-m\tildeb{D})_{|\tildeb{D}})\cong  H^1(\Proj^{n-1}, O_{\Proj^{n-1}}(n))=0.$$
Hence $H^{0}\bi(D, R^{1}\pi_{*}(L(-m\tildeb{D})_{|\tildeb{D}})\bi)=0.$

To prove $H^1\bi(D, \pi_{*}(L(-m\tildeb{D})_{|\tildeb D})\bi)=0,$ note that $\pi_{*}(L(-m\tildeb{D})_{|\tildeb D})=S^n \Omega^1_{\Xb |D}(-mD_{|D})$. 
Consider the filtration of $S^n \Omega^1_{\Xb |D}$ obtained by the exact sequence \ref{Conormal}: 
$$S^n \Omega^1_{\Xb |D}=F^0 \supseteq F^1 \supseteq ... \supseteq F^{m}\supseteq F^{m+1},$$
with quotients
$$F^i/F^{i+1}\cong S^{m-i} \Omega_{D}^{1} (-i D_{|D})$$
for each $i.$ 
Tensoring the filtration by $-m D_{|D},$ we get a filtration for $S^n \Omega^1_{\Xb |D}(-m D_{|D})$ whose successive quotients are $S^{m-i} \Omega_{D}^{1} (-j D_{|D})$ for some $j>0.$ As $D$ is an \'etale quotient of abelian variety, there exists a finite \'etale map $f:D' \to D,$ where $D'$ is an abelian variety. Since $f^* \Omega^1_{D} \cong \Omega^1_{D'},$ and $\Omega^1_{D'}$ is trivial, we get that $f^* S^{m-i} \Omega_{D}^{1} (-j D_{|D})$ is a power of $f^* (-j D_{|D}).$ Since the canonical bundle $K_{D'}$ is trivial, $f$ is finite and $-D_{|D}$ is ample, Kodaira vanishing theorem gives that $H^1(D',f^*( -j D_{|D}))=0$ for every positive integer $j.$

It follows that the successive quotients of the filtration of $S^n \Omega^1_{\Xb |D}(-mD_{|D})$ have vanishing first cohomology. Hence, $H^1\bi(D, S^n \Omega^1_{\Xb |D}(-mD_{|D})\bi)=0,$ i.e., 
$$H^1\bi(D, \pi_{*}(L(-m\tildeb{D})_{|\tildeb D})\bi)=0.$$
\end{proof}

Applying \cite[Example 2.1.29]{lazarsfeld1} to Theorem \ref{CotPos}, we get that symmetric differentials over $\Xb$ forms a finitely generated $\C$-algebra provided that the uniform depth is sufficiently large. More precisely, we get:  
\begin{Corollary}(Corollary \ref{fg})
With the same assumption as Theorem \ref{CotPos}, the graded ring $$\bigoplus_{n>0}H^{0}(\Xb,S^{n}\Omega^1_{\Xb})$$ is finitely generated $\C$-algebra.  
\end{Corollary}

\section{Application to Hyperbolicity of subvarieties}
We follow the same notations as previous section and denote a ball quotient $\Gamma\backslash\mathbb{B}^n$ with a torsion-free lattice $\Gamma$ by $X$, the boundary divisor $\Xb\setminus X$ by $D$, the toroidal compactification of $X$ by $\Xb$ and the uniform depth of cusps by $d$. In addition, we
assume that $\Xb$ does not have any orbifold point. Let $V$ be an irreducible subvariety of $\Xb$ intersecting $X.$

The purpose of this section is to prove the result related to hyperbolicity of $V.$ We first prove Corollary \ref{HyperVarIn} and Corollary \ref{GenSubInt}. Further, we prove Corollary \ref{VolIn} which tells us that the hyperbolicity increases in towers of normal covering in the sense that the minimum volume of subvarieties of $\Xb$ intersecting both $X$ and $D$ tends to infinity. Note that, as the components of $D$ are \'etale quotient of abelian varieties, we can not expect that the hyperbolic volume of a variety entirely contained in $D$ tends to infinity.

\begin{Corollary} \label{HyperVar} (Corollary \ref{HyperVarIn}) Suppose $V$ is smooth with dimension $m>0.$ Then, $\Q$-line bundle  
$$K_V-(r -1) D_{|V}$$
is ample for all rational $r \in (0,\frac{md}{4\pi}).$ Moreover, if $d>4\pi$ or $m>6$, then $K_{V}$ is ample.

\end{Corollary}

\begin{proof}
Let $a$ and $b$ be positive integers such that $a/b<d/4\pi$ and let $D'$ be $D_{|V}.$ Since the vector bundle 
 $S^{b}\Omega^{1}_{V}(\log(D' )) \otimes O_{V} (-a D')$ is a quotient of 
 $$\B( S^{b}\lgcot \otimes O_{\Xb} (-a D)\B)_{|V},$$
Theorem \ref{PositveLog} gives that  $S^{a}\Omega^{1}_{V}(\log(D'  )) \otimes O_{V} (-b D')$ is ample and therefore its determinant 
$$\frac{b}{m}\binom{m+b-1}{b} \bi (K_{V}(D')\bi) - \binom{m+b-1}{b} aD'$$
is ample. Scaling yields the first part of the claimed result. The rest follows from the fact that the uniform depth of cusps $d$ is at least $2.$        
    
\end{proof}
\begin{Corollary}\label{GenSub} (Corollary \ref{GenSubInt})
All subvarieties of $X$ are of general type provided that the uniform depth of cusps is greater than $4 \pi.$ 
\end{Corollary}
\begin{proof}
  Let $V_{0}$ be an $m$-dimensional (possibly  non-smooth) subvariety of $\Xb$ not entirely contained in $D$ and let $\mu : V' \rightarrow V_{0}$ be a resolution of singularities. There is a generically
surjective homomorphism $\mu^{*} \Omega^{m}_{\Xb} \rightarrow \Omega^{m}_{V'}=K_{V'}.$ Thanks to Theorem \ref{CotPos}, $\Omega^{m}_{\Xb}$ is ample modulo $D$ and therefore, the pull back $\mu^{*} \Omega^{m}_{\Xb}$ is ample modulo $\mu^{*}D$. This implies that $K_{V'}$ is in particular big, i.e., $V_{0}$ is of general type. 
\end{proof}

 The volume of a line bundle is a birational invariant and it is positive if and only if the line bundle is big. It turns out that if $L$ is nef, then $\operatorname{vol}(L)= L^{ n}.$

 In the setting of the compactification of a locally symmetric domain, a natural quantity reflecting the hyperbolicity behaviour is the volume of the log-canonical bundle. In particular, $vol_{V}(K_V+D_{|V})>0$ if and only if $V$ is of log-general type. In the case of ball quotients, Corollary \ref{neflog} in particular implies $\Omega^{1}_{V}(\log D_{|V})$ is nef and therefore $K_{V}+D_{|V}$ is nef on $V.$ Hence, 
$$\operatorname{vol}_{V}(K_V+D_{|V})=(K_{V}+D_{|V})^{m}.$$  

We show that the minimum volume of log-canonical bundle of subvarieties of $\Xb$ is controlled by the uniform depth of cusps:

\begin{Corollary} \label{Vol}
 Suppose $V$ is a $m$-dimensional smooth subvariety of $\Xb$ not entirely contained in $D$.  If $l$ is the number of component of $D$ intersecting with $V$, then
$$\operatorname{vol}_{V}(K_{V}+D_{|V}) > m^m  \Big(\frac{d}{4\pi}\Big) ^m l(m-1)!.$$

\end{Corollary}
\begin{proof}

Let $r= \frac{d}{4\pi}$ and $D'$ be $D_{|V}.$
Since $D$ is a union of \'etale quotient of abelian varieties with ample conormal bundle $O_{D}(-D)$, we get 

\begin{align}\label{DInq}
l\le \frac{D'\cdot (-D')^{m-1}}{(m-1)!}=\dfrac{-(-D')^m}{(m-1)!}.
\end{align}
On the other hand, Corollary \ref{HyperVar} implies that
$K_V -(mr-1) D'$
is in particular big and nef. Thus, $(K_{V}-(mr-1)D'\bi)^{m} >0 .$ 
On the other hand, $(K_{V}+D')_{|D'}$ is nef and $-D'_{|D'}$ is ample. Consequently, for every $0<i<m,$
$$(K_{V}+D')^{m-i}(-D')^{i }=- \bi((K_{V}+D')_{|D'} \bi) ^{m-i}(-D'_{|D'})^{i-1 } $$
is non-positive. Hence, $(K_{V}-(mr-1)D'\bi)^{m} >0 $ yields that 
 
    $$(K_{V}+D')^m > - (rm) ^ m (-D')^m. $$ 
Combining this with the inequality \ref{DInq} gives the desired inequality.
\end{proof}

\begin{Remark}
 Parker's generalization of Shimizu's lemma \cite[Proposition 2.4.]{parker1998volumes} gives that the uniform depth of cusps is at least $2$. Plugging in this result to Corollary \ref{Vol} gives
 that if $V$ intersects $D$, then $V$ is of log general type. This can be concluded from the recent result of Guenancia \cite[Theorem B]{guenancia2022quasiprojective} as well. 
\end{Remark}


We also get a uniform lower bound for the volume of the canonical bundles of subvarieties:
\begin{Corollary}\label{CanonicalVol}
If the canonical depth of cusps d is greater than $4\pi$, then 
$$\operatorname{vol}_{V}(K_{V}) \ge \bi( \frac{md}{4\pi}  -1 \bi )^m l(m-1)!$$
\end{Corollary}
\begin{proof}
 Since $d > 4\pi,$ by Corollary \ref{HyperVar} we get that $K_V$ is ample and therefore $\operatorname{vol}_{V}(K_V) = K_{V}^m.$ Let $r= \frac{d}{4\pi} $ and $D'$ be $D_{|V}.$ As sated in the proof of Corollary \ref{Vol}, $(K_V-(mr-1)D')^m>0$. Since $K
_{V|D'}$ and $-D'_{D'}$ are ample, for every $0<i<m,$

$$K_V^{m-i}\cdot (-D')^i \cong - (K_{V|D'})^{m-i} \cdot (-D'_{|D'})^{i-1} $$
is negative. Hence, $(K_V-(mr-1)D')^m>0$ together with the inequality \ref{DInq} yields that 
$$K_V^m \ge -(rm-1)^m(-D')^m \ge (rm-1)^m l (m-1)!$$
\end{proof}

 As a generalization of Brunebarbe's work \cite{brunebarbe2020increasing}, for towers of ball quotients we show that:
 
\begin{Corollary} (Corollary \ref{VolIn})
Let $\{X_i\}_{i=1}^{\infty}$ be a tower of $X=X_1.$ 
Suppose that toroidal compactification of $X_i$ does not have any orbifold points. Then, given a positive number $v$, for all but finitely many $i,$ every subvariety $V$ of $X_i$ containing a cusp of $X_i$ have $vol(K_V)>v.$      

\end{Corollary}
\begin{proof}
Combining Lemma \ref{tower} and Corollary \ref{CanonicalVol} gives the result.   
 
\end{proof}

\bibliographystyle{alpha}
\bibliography{Ref}

\end{document}